\documentclass[11pt]{amsart}
\usepackage{amsmath,amsthm, amscd, amssymb, amsfonts, mathrsfs,mathtools}
\usepackage[all]{xy}
\usepackage{color}
\usepackage[inline]{enumitem}
\usepackage{bm}
\usepackage{graphicx}
\usepackage{hyperref}

\usepackage{mathrsfs}

\usepackage{hhline} 
\usepackage[active]{srcltx}%{vista inversa}

\newcommand{\PSL}{\mathbf{PSL}}
\newcommand{\Gb}{\boldsymbol{G}}
\newcommand{\PSp}{\mathbf{PSp}}

\newcommand{\Pom}{\mathbf{P\Omega}}
\newcommand{\PSU}{\mathbf{PSU}}

\newcommand{\SU}{\mathbf{SU}}

\newcommand{\Os}{\mathscr{O}}

\newcommand{\Y}{{\tt Y}}
\newcommand{\Jf}{{\tt J}}

\newcommand{\letra}{\mathtt d}

\newcommand{\Ot}{\mathfrak{O}}
\newcommand{\ct}{\mathtt{c}}
\newcommand{\x}{\mathtt{x}}
\newcommand{\y}{\mathtt{y}}
\newcommand{\z}{\mathtt{z}}

\newcommand{\st}{\mathtt{s}}
\newcommand{\tu}{\mathtt{u}}

\newcommand{\St}{\mathtt{S}}

\newcommand{\blambda}{\pmb{\lambda}}

\newcommand\sn{\mathbb S_n}

\newcommand{\diag}{\operatorname{diag}}

\newcommand{\Imm}{\operatorname{Im}}

\newcommand{\Inn}{\operatorname{Inn}}

\newcommand{\kc}{\mathbb F_q}

\newcommand{\trid}{\triangleright}

\newcommand{\Lb}{{\mathbb L}}

\newcommand{\kk}{\Bbbk}%{}

\newcommand{\ku}{\mathbb C}

\newcommand{\Z}{\mathbb Z}
\newcommand{\N}{{\mathbb N}}
\newcommand{\I}{{\mathbb I}}

\newcommand{\G}{{\mathbb G}}

\newcommand{\T}{{\mathbb{T}}}
\newcommand{\U}{\mathbb{U}}

\newcommand{\Pa}{\mathbb{P}}

\newcommand{\F}{{\mathbb F}}
\newcommand{\C}{{\mathcal C}}

\newcommand{\GL}{\mathbf{GL}}
\newcommand{\GU}{\mathbf{GU}}
\newcommand{\PGL}{\mathbf{PGL}}
\newcommand{\SL}{\mathbf{SL}}
\newcommand{\Sp}{\mathbf{Sp}}

\newcommand{\SO}{\mathbf{SO}}

\newcommand{\Fr}{\operatorname{Fr}}

\newcommand{\Le}{{\mathbb L}}

\newcommand{\Oc}{{\mathcal O}}
\newcommand{\oc}{{\mathcal O}}

\newcommand{\Ss}{\mathcal S}

\newcommand\normal{\lhd}

\numberwithin{equation}{section}

\theoremstyle{plain}

\newtheorem{lema}{Lemma}[section]
\newtheorem{theorem}[lema]{Theorem}
\newtheorem{maintheorem}{Theorem}

\newtheorem{prop}[lema]{Proposition}
\newtheorem{question}{Question}

\newtheorem{claimpsp}{Claim}

\newtheorem{claimsln}{Claim}

\newtheorem{question-app}{Question}

\theoremstyle{definition}
\newtheorem{definition}[lema]{Definition}
\newtheorem{exa}[lema]{Example}

\theoremstyle{remark}
\newtheorem{obs}[lema]{Remark}
\newtheorem{rmk}[lema]{Remark}
\newtheorem{step}{Case}

\newcommand\wdot{\dot{w}}
\newcommand\id{\operatorname{id}}

\newcommand\an{\mathbb A_n}

\newcommand\ac{\mathbb A_4}
\newcommand\aco{\mathbb A_5}
\newcommand\as{\mathbb A_6}

\newcommand\s{\mathbb S}

\def\pf{\begin{proof}}
\def\epf{\end{proof}}

\newcounter{tabla}\stepcounter{tabla}

\begin{document}

\renewcommand{\baselinestretch}{1.2}

\thispagestyle{empty}

\title[Nichols algebras over semisimple classes]
{Finite-dimensional pointed Hopf algebras\newline over finite simple groups of Lie type VII. \newline 
Semisimple classes in $\PSL_{n}(q)$ and $\PSp_{2n}(q)$}

\author[N. Andruskiewitsch, G. Carnovale, G. A. Garc\'ia]
{Nicol\'as Andruskiewitsch, Giovanna Carnovale and\newline Gast\'on Andr\'es Garc\'ia}


\address{\hspace{-12pt}N. A.: FaMAF-Universidad Nacional de C\'ordoba, CIEM (CONICET),
Medina A\-llen\-de s/n, Ciudad Universitaria, 5000 C\' ordoba,  Argentina.} 
\email{nicolas.andruskiewitsch@unc.edu.ar }

\address{\hspace{-13pt} G. C.: Dipartimento di Matematica Tullio Levi-Civita, 
Universit\`a degli Studi di Padova, 
via Trieste 63, 3512,1 Padova, Italia.} 
\email{carnoval@math.unipd.it, +39-049-8271354}

\address{\hspace{-14pt} G. A. G.: Departamento de Matem\'atica, Facultad de Ciencias Exactas,
Universidad Nacional de La Plata. CMaLP-CIC-CONICET. Calle 47 y Calle 115, 1900 La Plata, Argentina.}
\email{ggarcia@mate.unlp.edu.ar}

\begin{abstract}
We show that the Nichols algebra of a simple Yetter-Drin\-feld module over a projective 
special linear group over a finite field whose support is a semisimple orbit has infinite dimension,
provided that the elements of the orbit are reducible; we obtain a similar result for  all semisimple orbits in a 
finite symplectic group except in low rank. 
We prove that orbits of irreducible elements in the projective  special linear groups
could not be treated with our methods. 
 We conclude that any
finite-dimensional pointed Hopf algebra $H$ with group of group-like elements isomorphic to $\PSL_{n}(q)$ $(n\ge 4)$,  $\PSL_{3}(q)$  $(q >2)$, or $\PSp_{2n}(q)$ $(n \ge 3)$, 
is isomorphic to a group algebra, completing work in \texttt{arXiv:1506.06794}. 
\end{abstract}

\maketitle

\setcounter{tocdepth}{1}
\tableofcontents

\section{Introduction}
\subsection{The problem}
Let $G$ be a finite group. The conjugacy class of $x \in G$ is denoted by $\Oc^G_x$ or $\Oc_{x}$. 
The subgroup of $G$ generated by $I \subset G$ is denoted by $\langle I \rangle$. 
Consider the following properties of a conjugacy class $\Oc$ of $G$:

\begin{enumerate}[leftmargin=*,label=\rm{(\Alph*)}]\setcounter{enumi}{2}
\item\label{item:typeC-group} There are $H<G$ and $r,s\in H\cap\Oc$ such that
$rs\neq sr$, $H=\langle \Oc_r^H,\Oc_s^H\rangle$, $\Oc_r^H\neq \Oc_s^H$ and
$\min \{\vert\Oc_r^H\vert, \,\vert\Oc_s^H\vert\}>2$ or $\max \{\vert\Oc_r^H\vert, \,\vert\Oc_s^H\vert\}>4$. 

\medbreak
\item\label{item:typeD-group} There exist $r$, $s\in \Oc$
such that $\Oc_r^{\langle r, s\rangle}\neq \Oc_s^{\langle r, s\rangle}$ and $(rs)^2\neq(sr)^2$.

\medbreak \setcounter{enumi}{5}
\item\label{item:typeF-group}
There are $r_a\in \Oc$, $a\in \I_4 = \{1,2,3,4\}$,
such that $\Oc_{r_a}^{\langle r_a: a\in \I_4\rangle}\neq \Oc_{r_b}^{\langle r_a: a\in \I_4\rangle}$  and
$r_a r_b \neq r_br_a$ for $a  \neq b\in \I_4$.
\end{enumerate}

\medbreak
We say that $\Oc$ is of type C, D, F when the corresponding property holds. 
As explained in the Introduction to \cite{ACG-V}, the next question has profound implications 
for the classification of finite-dimensional pointed Hopf algebras.

\begin{question}\label{question:collapse}
Determine which conjugacy   classes of a given finite (non-abelian) group $G$ are of type C, D or F.
\end{question}

Indeed, if a conjugacy   class $\Oc$ is type C, D or F, then any Nichols algebra of group type with support isomorphic
to $\Oc$ has infinite dimension; for brevity we say in this case that $\Oc$ collapses. For the purposes of this paper 
further precision on Nichols algebras is not needed. 

\medbreak
If $\Oc$ is neither of type C, D nor F then we say that it is \emph{kthulhu}.
It follows at once from the previous definitions that if $\Oc \cap H$ is either abelian  or a single conjugacy class of $H$
for any $H\leq G$, then $\Oc$ is kthulhu.

Intuitively, the criteria of types C, D and F are inductive arguments that are more flexible in the language of racks,
see Section \ref{sec:racks}. Conjugacy classes are the prototypical examples of racks.
One may wonder whether there are other inductive arguments that force the collapse of a conjugacy class.
In this respect we say that a rack is \emph{sober} if  every subrack is either abelian or indecomposable \cite[\S 1.5]{ACG-I}; 
and   \emph{austere} if every subrack generated by two elements is either abelian or  indecomposable \cite[2.11]{ACG-III}.
Clearly, sober implies austere and austere implies kthulhu. 
Subsidiary to Question \ref{question:collapse}, we  propose:

\begin{question}\label{question:sober}
Determine which conjugacy   classes of a given finite (non-abelian) group $G$ are sober or austere, or kthulhu.
\end{question}

\subsection{Simple groups of Lie type}
It is natural, and convenient, to start addressing Questions \ref{question:collapse} and \ref{question:sober} by assuming that $G$ is simple non-abelian, see \cite{AG} for the importance of this reduction. When $G$ is alternating or sporadic, this was addressed in
\cite{AFGV-ampa,AFGV-espo,F,FaV}. 
The series of papers \cite{ACG-I,ACG-II,ACG-III,ACG-IV,ACG-V,CC-VI} treat the case when $G$ is simple of Lie type.
In the first five papers, Questions  \ref{question:collapse} and \ref{question:sober} were answered for non-semisimple conjugacy classes in Chevalley or Steinberg groups. The sixth is devoted to Suzuki and Ree groups.

\medbreak
In the present paper we deal with semisimple conjugacy classes in the classical Chevalley groups.
The main difficulty  is due to the deeper influence of arithmetics, 
as opposed to the unipotent classes and the mixed classes,  which can be reduced in most of the cases 
to a unipotent one in a smaller group.  
We summarize our main results and then discuss the  proofs.

\begin{maintheorem}\label{thm:main}
Let $\Oc \neq \{e\}$  be a semisimple conjugacy class in a group $\Gb$.
\begin{enumerate}[leftmargin=*,label=\rm{(\roman*)}] 
\item\label{item-main-psl} If $\Gb = \PSL_{n}(q)$,  then any  $\Oc$ not listed in Table \ref{tab:ss-psl2}  collapses.

\item\label{item-main-psp}
If $\Gb = \PSp_{n}(q)$, then any $\Oc$ collapses with the possible exception of the orbit of non-trivial involutions if $n=2$ and $q\in \{ 3,5,7\}$.
\end{enumerate}
\end{maintheorem}

Part \ref{item-main-psl} is proved in Theorem \ref{thm:section-SL-summary}; see Remark \ref{rmks:sl2-ss} for the 
cases with small $q$ excluded in the statement.
Part \ref{item-main-psp} is proved in Theorem \ref{thm:conclusion-semisimple-SP}.
For other Chevalley groups, there is substantial information in Theorems \ref{thm:split-collapses} and \ref{thm:section-meetsK-PSO}.

\medbreak
Roughly the proofs of these results go  as follows: pick a simple group $\Gb$, a surjective morphism of
groups $\pi: G \to \Gb$ and conjugacy classes $\Ot$ and $\Oc$ in $G$ and $\Gb$ respectively, such that 
$\pi(\Ot) = \Oc$.
Then look at the subgroups $H$ intersecting $\Ot$. If $H\cap \Ot$ 
splits as more than one conjugacy class of $H$ for one  $H \leq G$, then work out the 
details to have that $\Ot$ is of type C or D
and that  this is preserved by the projection $\pi:\Ot \to \Oc$. 
When $G$ is of Lie type, the subgroup $H$ is usually found by looking in one way or another at the 
structure of the algebraic group behind $G$.

But if $\pi(H)\cap \Oc$ is either abelian or 
one conjugacy class of $\pi(H)$ for every $H \leq G$, then  $\Oc$ is kthulhu.
When this  happens, usually  $\Gb$ is `small' and has few subgroups. 

\medbreak
In the present paper, we found that it also happens to any conjugacy class $\Oc$
of an irreducible element in $\PSL_{n}(q)$ where $n$ is an odd prime. 
To show this we used the main result of \cite{GPPS} 
to get the list of the $H\leq G$ intersecting $\Ot$ together with some arithmetic manipulations.
This outcome differs significantly with the results of the previous of the series 
and underlines the connection  of semisimple classes with arithmetical aspects.

\begin{table}[ht]
\caption{Kthulhu semisimple classes in $\PSL_n(q)$.}\label{tab:ss-psl2}
\begin{center}
\begin{tabular}{|c|c|c|c|}
\hline $n$ & $q$  & class &  Remark  \\
\hline  \hline
   &   $2$, $\PSL_{2}(2) \simeq \s_3 $  &  $(3)$      & 
abelian 
\\ \cline{2-4}
   &   $3$, $\PSL_{2}(3) \simeq \ac$  &  $(2^2)$      &  
abelian
\\ \cline{2-4}
   &   $4$,  $\PSL_{2}(4) \simeq \aco$    &  $(5)$      & sober
\\ \cline{2-4}
2   &   $5$, $\PSL_{2}(5) \simeq \aco$  &  $(1, 2^2)$      &  sober
\\ \cline{2-4}
 & $9$, $\PSL_{2}(9) \simeq \as$ &   $(1,5)$ & kthulhu
\\ \cline{2-4}
 & even 
 and  not a square & irreducible, order 3 & sober\\
\cline{2-4} 
 & all
  & irreducible, order $> 3$ & sober  \\
\hline 
odd prime & all & irreducible  & kthulhu  \\
\hline
\end{tabular}
\end{center}
\end{table}

\subsection{Applications to Hopf algebras}
As in previous papers, we say  that a finite group $G$ 
collapses if every finite-dimensional pointed Hopf algebra $H$ with $G(H) \simeq G$ is necessarily
isomorphic to $\ku G$.  As a corollary of our main Theorem and results from previous papers in the series,
we obtain new families of groups that collapse, see Theorem \ref{thm:groups-collapse}, 
extending \cite[Theorem 1.2]{ACG-III}.
For this, we first draw the complete list of kthulhu classes in the simple groups  
$\PSL_{n}(q)$ and $\PSp_{2n}(q)$ for $n\geq2$, any $q$.  This information combines the main result of this paper with a corrected  version of \cite[Table 1]{ACG-I},  \cite[Table 3]{ACG-III} and \cite[Theorem 1.1]{ACG-V} and a careful analysis of the cases of the groups $\PSL_2(q)$ for $q=2,3,4,5,7,9$, where some exceptions occur.   We point out that the classes labeled $(1^{r_1},2)$ in $\Sp_{2n}(9)$, occurring in \cite[Tables 3,5]{ACG-III} are in fact not kthulhu: they are of type C,  as each of them includes a non-trivial unipotent class of type $(2)$ in $\PSL_2(9)\simeq\as$ which is of type C,  cf.  Example \ref{exa:3-cycles}.  Also, the classes of involutions in $\PSL_2(7)$ in \cite[Table 1]{ACG-III} are not not kthulhu: since $\PSL_3(2)\simeq\PSL_2(7)$, they are of type C by \cite[Lemma 2.12]{ACG-III}.

\begin{maintheorem}
Let $\Gb$ be either $\PSL_{n}(q)$ or $\PSp_{2n}(q)$, $n\geq2$ and let $\Oc$ be a non-trivial conjugacy class in $\Gb$ different from the class of a split involution in $\PSp_4(7)$. Then 
 $\Oc$ is kthulhu if and only if it occurs in Table \ref{tab:kthulhu-simplectic}.
\end{maintheorem}
\begin{table}[ht]
\caption{\small Kthulhu  classes, $\Gb = \PSL_{n}(q)$ or $\PSp_{2n}(q)$, $n\geq2$.}\label{tab:kthulhu-simplectic}
\begin{center}
\begin{tabular}{|c|c|p{2cm}|c|c|}
\hline $\Gb$ &$n$  & $q$ & type of class & description/label\\  
\hline\hline
&  & even or else
 odd and not a square & unipotent & $(2)$  \\
 \cline{3-5}
 & &$5$& semisimple & involution  \\
 \cline{3-5}
$\PSL_{n}(q)$ &$2$ 
 & all & semisimple&  irreducible, $|x|>3$   \\
 \cline{3-5}
 &  & even and not a square & semisimple & irreducible $|x|=3$\\
 \cline{2-5}
& $3$ & $2$ & unipotent &$(3)$ \\
 \cline{2-5}
&
odd prime & all &  semisimple  &  irreducible   \\
\hline
&   & even & unipotent & $W(1)^a\oplus V(2)$    \\
 \cline{3-5}
$\PSp_{2n}(q)$& $\geq 2$ & odd and not a square &   unipotent & $(1^{r_1}, 2)$ 
 \\
 \cline{2-5}
&  & even & unipotent & $W(2)$    \\
 \cline{3-5}
& $2$ &3,5  & semisimple & split involution 
\\
\hline
\end{tabular}
\end{center}
\end{table}

It remains open to determine whether the conjugacy class of split involutions in $\PSp_4(7)$ is kthulhu. 

The next result combines \cite[Theorem 1.4]{ACG-I}, \cite[Theorem 1.2]{ACG-III}
 and \cite[Theorem 1.2]{ACG-V} with Theorem \ref{thm:main}.

\begin{maintheorem}\label{thm:groups-collapse}
The groups $\PSL_{n}(q)$ with $n\ge 4$,  $\PSL_{3}(q)$ with $q >2$, 
and $\PSp_{2n}(q)$, $n \ge 3$, collapse. \qed
\end{maintheorem}

In the group $\PSL_3(2)$,  there is just one class that could not be treated, namely
the regular unipotent class $\Oc$, which  is sober. 
Actually $\PSL_3(2) \simeq \PSL_{2}(7)$ and  for this group, $\Oc$ is semisimple.

\subsection{Conventions}\label{subsec:conventions}
If $a\leq b\in\N_0$, then $\I_{a,b}$ denotes $\{a,a+1,\,\ldots,b\}$; also 
$\I_a = \I_{1,a}$ for simplicity. 

Let $G$ be a group. The centraliser of $X\subset G$ is denoted by $C_G(X)$.
If $x,y\in G$, then $x \trid y\coloneqq xyx^{-1}$. We write $X \ge Y$, or $Y \le X$, to express that $Y$ is a subrack of $X$ (or a subgroup,
or more generally a subobject in a given category). The normality of a subgroup is expressed by $N\lhd G$.

Let $q =p^m$, $p$ a prime and $m \in \N$. 
Let $\kc$ be the field with $q$ elements and $\kk$ the algebraic closure of $\kc$.
We denote by $\G_n(\mathbf k)$ the group of $n$-th roots of unity in a field $\mathbf k$.

\section*{Acknowledgements}
We thank Gunter Malle for very helpful email exchanges and Andrea Lucchini for pointing out several references.

N. A. was  partially supported by CONICET (PIP  11220200102916CO),
FONCyT-ANPCyT (PICT-2019-03660) and Secyt (UNC). 
G. A. G. was partially supported by CONICET (PIP 11220200100423CO), Secyt (UNLP) and FONCyT-ANPCyT (PICT-2018-00858).
G. C. was partially supported by Projects BIRD179758/17, DOR2207212/22, and BIRD203834 of 
the University of Padova. The results were obtained during visits of N. A. 
to the University of Padova, and of G. C. to the University 
of C\'ordoba,  partially supported by the bilateral agreement between 
these  Universities and the INdAM-GNSAGA Visiting Professor program.

\bigbreak

\section{ Racks}\label{sec:racks}

\subsection{Racks}\label{subsec:racks}
As in previous papers we use the language of racks; see \cite{AG} for more information. A rack is a pair $(\Oc, \trid)$
where $\Oc$ is a non-empty set and $\trid: \Oc \times \Oc \to \Oc$ is a self distributive operation
such that $\varphi_x \coloneqq x\trid \underline{\,\,\;}$ is bijective for any $x\in \Oc$. 
A subset $\Oc'\subset\Oc$ is a subrack if $\Oc'\trid\Oc'\subset\Oc'$.
Let $\Inn \Oc$ be the group generated by the image of the map $\varphi:\Oc \to \mathbb S_{\Oc}$.
The main examples of racks considered in this paper
are (unions of) conjugacy classes of a finite group with $\trid$ being the conjugation.
A rack $(\Oc, \trid)$ is abelian if $x\trid y = y$ for any $x,y\in \Oc$. 
Also, a rack is indecomposable if it can not be presented as the disjoint union of two subracks and decomposable otherwise.
\medbreak
The following observation will be useful, especially  when dealing with orthogonal groups.
\begin{obs}\label{obs:normal-subrack} Let $G$ be a finite group,
 $N\lhd G$,  $g\in G - N$ and  $\Oc_g^N$  the orbit of $g$ under the conjugation action of $N$.
 Then  $\Oc_g^N$ is a subrack  of $\Oc_g^G$. 
\end{obs}

This is a special case of \cite[Remark 3.2]{CGI} that can be verified directly. 
Notice that if $N\leq G$ is not normal, then $\Oc_g^N$  may fail to be a subrack of $\Oc_g^G$. For instance, take $G={\mathbb S}_4$, $g=(123)$ and $N =\langle (12)(34)\rangle$; then $\Oc_{g}^N=\{(123),(142)\}$  is not closed under the rack operation.

\subsection{Racks of type C}\label{subsec:typeC}
The following notion was introduced in  \cite[Definition 2.3]{ACG-III}
motivated by \cite[Theorem 2.1]{HV}.

\begin{definition}\label{def:rack-typeC} A rack
$X$ is \emph{of type C}  if there are  a decomposable subrack
$Y = R\coprod S\leq X$,  $r\in R$ and  $s\in S$ such that $r \triangleright s \neq s$,
\begin{align*}
R &= \Oc^{\Inn Y}_{r}, \qquad S = \Oc^{\Inn Y}_{s},
& &\min \{\vert R \vert, \vert S \vert \}  > 2  \text{ or } 
\max \{\vert R \vert, \vert S \vert \}  > 4.
\end{align*}
\end{definition}

The group-theoretical reformulation \ref{item:typeC-group} of  the definition of type C  is \cite[Lemma 2.8]{ACG-III}.
We need a variation of \cite[Lemma 2.8]{ACG-III} in order to encompass the situation in Remark \ref{obs:normal-subrack}.  The proof can be repeated verbatim: we recall it here for completeness.  

\begin{lema}\label{lem:equivC} Let $G$ be a finite group, $g\in G$ and $N\normal G$.
The orbit  $\Oc=\Oc_g^N$  is  of type C
if and only if there are $H \leq \langle\Oc\rangle$, $r,s\in\Oc\cap H$ such that
\begin{align}
\label{eq:equivC2} &\Oc_r^H\neq \Oc_s^H; \\
\label{eq:equivC1}
&rs\neq sr;\\ 
\label{eq:equivC3} &H=\langle \Oc_r^H,\Oc_s^H\rangle;\\
\label{eq:equivC4} & 
\min \{\vert\Oc_r^H\vert, \,\vert\Oc_s^H\vert\}>2 \quad\text{ or }\quad
\max \{\vert\Oc_r^H\vert, \,\vert\Oc_s^H\vert\}>4.
\end{align} 
\end{lema}
\begin{proof}
Assume that $r,\,s$ and $H$ are as above and set $R\coloneqq\Oc_r^H$ and $S\coloneqq\Oc_s^H$. 
If  $r'=h\trid r\in \Oc_r^H=R$ for some $h\in H$, then there exist $x_1,\cdots,\,x_k\in \Oc$ 
such that $h=x_1\cdots x_k$.  Hence 
\begin{align*}r'=x_1\trid(x_2\trid(\cdots (x_k\trid r))\in \Oc(\trid \Oc(\trid \cdots \Oc\trid \Oc)))\subset \Oc,\end{align*} so 
$R\subset \Oc\cap H$ and similarly $S\subset \Oc\cap H$. By \eqref{eq:equivC2}  the subset 
$Y\coloneqq R\coprod S\subset \Oc$ is a decomposable subrack, and $r\trid s\neq s$ is \eqref{eq:equivC1}.  In addition, 
\begin{align*}
\Oc_r^{\Inn Y}=\Oc_r^{\langle R, S\rangle}=\Oc_r^H=R, 
\end{align*}
where the second equality follows from \eqref{eq:equivC3}, and similarly, $\Oc_s^{\Inn Y}=S$. 
The estimate on $R$ and $S$ is \eqref{eq:equivC4}. Hence $\Oc$ is of type C.

Conversely, let $X=\Oc$ and $r,\,s,\,R,\,S,\,Y$ be as in Definition \ref{def:rack-typeC}. Setting $H\coloneqq\langle R,\,S\rangle$, 
we immediately have  \eqref{eq:equivC3}, $H\leq \langle \Oc\rangle$, $R= \Oc_{r}^{H}$, $S= \Oc_{s}^{H}$ 
and  $r,\,s\in \Oc\cap H$. Finally \eqref{eq:equivC1}, \eqref{eq:equivC2} and \eqref{eq:equivC4} are straightforward. 
\end{proof}

\begin{obs}\label{obs:equivC} 
Let $G$, $N$ and $\Oc$ be as in Lemma \ref{lem:equivC}.  
\begin{enumerate}[leftmargin=*,label=\rm{(\alph*)}]
\item\label{item:equivC1} If  there exist $r,\,s\in \Oc$ satisfying \eqref{eq:equivC1}, 
$\Oc_r^{\langle r,s\rangle}\neq \Oc_s^{\langle r,s\rangle}$ and:
\begin{align*}
\min \{\vert\Oc_r^{\langle r,s\rangle}\vert, \,\vert\Oc_s^{\langle r,s\rangle}\vert\}>2 \quad\text{ or }\quad
\max \{\vert\Oc_r^{\langle r,s\rangle}\vert, \,\vert\Oc_s^{\langle r,s\rangle}\vert\}>4,
\end{align*}
then $\Oc$ is of type C by Lemma \ref{lem:equivC} applied to $H\coloneqq\langle r,\,s\rangle$. 

\item\label{item:equivC2} If $|g|$ is odd and $r,\, s\in \Oc$, then for any  $H\leq G$ containing $r$ and $s$ the estimate \eqref{eq:equivC4} follows from \eqref{eq:equivC1},  since then $|s|=|r|=|g|$ is odd, whence $\vert\Oc_r^H\vert\geq \vert\Oc_r^{\langle s\rangle}\vert\geq 3$, and similarly for $\Oc_s^H$. This generalizes \cite[Lemma 2.7 (b)]{ACG-III} to  the situation of Remark \ref{obs:normal-subrack}.
\end{enumerate}
\end{obs}

\begin{exa}\label{exa:n-ciclo}
Let $n\geq 5$ be odd.  Then the conjugacy class $\Oc$ of $n$-cycles in $\sn$ is of type C. Indeed, $\Oc$ splits into two classes $\Oc'$ and $\Oc''$ in $\an$ and $|\Oc'|=|\Oc''|=\frac{(n-1)!}{2}>n$ elements. Therefore, if $r\in\Oc'$, there exists $s\in\Oc''$ such that $s\not\in C_{\an}(r)=\langle r\rangle$ and the result follows from  Remark \ref{obs:equivC}.
\end{exa}

\begin{exa}\label{exa:1122}
The class $\Oc$ corresponding to the partition $(1^2,2^2)$ in $\as$ is of type C.   Indeed,  $H:=C_{\as}(56)\simeq \s_4$ and $H\cap \Oc$ contains all involutions of the form $(ab)(cd)$ for $a,\,b,\,c,\,d\notin\{5,\,6\}$ and those of the form $(ab)(56)$ for $a,\,b\notin\{5,\,6\}$. Therefore, $|\Oc'\cap H|=12$ and $\Oc'$ contains all involutions in $H$.  
Now,  the involutions in $\s_4$ are parted into two classes of size $6$, and $\s_4$ contains non-commuting non-conjugate involutions.  Hence,  we can find $r,\,s\in H\cap \Oc' $ such that $r\trid s\neq s$ and $\Oc_r^H\neq\Oc_s^H$, with $|\Oc_r^H|=|\Oc_s^H|=6$.  Finally,  $\langle \Oc_s^H,\,\Oc_ r^H\rangle=H$ because $\s_4$ is generated by its involutions.  We conclude by Lemma \ref{lem:equivC}.
\end{exa}

\begin{exa}\label{exa:3-cycles}
The class  of $3$-cycles in $G={\mathbb A}_3$ or ${\mathbb A}_4$ is kthulhu because its intersection with any subgroup of $G$ is either abelian or a conjugacy class.

The class $\Oc$ of $3$-cycles in ${\mathbb A}_n$ for $n\geq5$ and the class $\Oc'$ labeled  $(3^2)$ in ${\mathbb A}_6$ are of type C.  Indeed,   $\Oc \cap \ac$ splits into the classes $\Oc_{(123)}$ and $\Oc_{(124)}$.  Since the representatives do not commute, $\Oc$ is of type C by Remark \ref{obs:equivC}.  Any non-inner automorphism of ${\mathbb S}_6$ interchanges $\Oc$ and $\Oc'$ in  ${\mathbb A}_6$, so $\Oc'$ is of type C as well.
\end{exa}

Here is an easy but useful application of the previous Lemma.

\begin{lema}\label{lem:typeC-subgroup}
Let $G$ be a finite group, $H \leq G$, $x \in H$. Assume that 
\begin{align}\label{eq:typeC-subgroup1}
H & \text{ is not abelian};\\ 
\label{eq:typeC-subgroup2} 
H &= \langle \Oc_{x}^{H}\rangle; \\
\label{eq:typeC-subgroup3} \text{there exists } s& \in \Oc_{x}^{G}\cap H:\ s\notin  \Oc_{x}^{H}, \ 
\vert\Oc_{s}^{H}\vert > 2.
\end{align}
Then $\Oc_{x}^{G}$ is of type C.
\end{lema}

\pf There is $r\in \Oc_{x}^{H}$ such that $rs \neq sr$; otherwise $s \in Z(H)$ by  \eqref{eq:typeC-subgroup2}, hence 
$\vert\Oc_{s}^{H}\vert =1$. Thus \eqref{eq:equivC1} holds and 
\eqref{eq:equivC2} and \eqref{eq:equivC3} are clear by construction.
Finally, $\vert \Oc_r^H \vert > 2$ by \eqref{eq:typeC-subgroup1}, thus \eqref{eq:equivC4} holds.
\epf

Here is another way to detect racks of type C.

\begin{lema}\label{lem:typeC-directproduct}
Let $G_1$ and $G_2$ be finite groups, $a_1\neq b_1 \in G_1$, $a_2, b_2 \in G_2$.
Set $r = (a_1, a_2), s = (b_1, b_2) \in G\coloneqq G_1 \times G_2$.
Assume that 
\begin{align}\label{eq:typeC-directproduct1}
a_1 b_1 &= b_1a_1, & a_2 b_2& \neq b_2a_2;\\ 
\label{eq:typeC-directproduct2} 
\Oc_{a_1}^{G_1} &= \Oc_{b_1}^{G_1}, & \Oc_{a_2}^{G_2} &= \Oc_{b_2}^{G_2}; \\
\label{eq:typeC-directproduct3} G_2 &= \langle \Oc_{a_2}^{G_2}\rangle. 
\end{align}
Then $\Oc_{r}^{G} = \Oc_{a_1}^{G_1} \times \Oc_{a_2}^{G_2}$ is of type C.
\end{lema}

\pf Let $H = \langle \{a_1\} \times \Oc_{a_2}^{G_2}, \{b_1\} \times \Oc_{a_2}^{G_2}\rangle \ni r,s$;
\eqref{eq:equivC1} is evident. We claim that
$\Oc_r^H = \{a_1\} \times \Oc_{a_2}^{G_2}$.
Indeed, $\subseteq$ follows from \eqref{eq:typeC-directproduct1}, and 
$\supseteq$  from \eqref{eq:typeC-directproduct3}:
\begin{align*}
y\in G_2  &\implies \exists x_1, \dots, x_t \in \Oc_{a_2}^{G_2}: y = x_1 \dots x_t\\
&\implies (a_1, y\trid a_2) = (a_1, x_1) \trid \left(  (a_1, x_2)\trid \dots (a_1, a_2) \right) \in \Oc_{r}^{H}.
\end{align*}
Similarly, $\Oc_s^H = \{b_1\} \times \Oc_{b_2}^{G_2}$.
Hence   \eqref{eq:equivC2} and \eqref{eq:equivC3} follow.
Finally, if $\vert \Oc_r^H \vert = \vert \Oc_s^H \vert = \vert \Oc_{a_2}^{G_2} \vert \leq 2$, then $a_2$ and $b_2$ commute, contradicting 
\eqref{eq:typeC-directproduct1}.
\epf

\begin{obs}\label{obs:typeC-directproduct}
Let $G_1$ and $G_2$ be finite groups, $a_1\neq b_1 \in G_1$.
The hypotheses of  Lemma \ref{lem:typeC-directproduct}
on $G_2$ hold when $G_2/Z(G_2)$ is a non-abelian simple group 
and $a_2\in G_2$ is not central. Namely, $\langle \Oc_{a_2}^{G_2}\rangle \lhd G_2$,
hence it is  all of $G_2$ giving \eqref{eq:typeC-directproduct3}. 
Furthermore there is $b_2 \in \Oc_{a_2}^{G_2}$ that does not commute with $a_2$, 
because $G_2$ is not abelian, as needed in \eqref{eq:typeC-directproduct1}.
\end{obs}

\subsection{Racks of type D}\label{subsec:typeD-F}
A rack $X$ is \emph{of type D}
if it has  a decomposable subrack
$Y = R\coprod S$ with elements $r\in R$, $s\in S$ such that
$r\trid(s\trid(r\trid s)) \neq s$ \cite[Definition 3.5]{AFGV-ampa}. 
If $\Oc$ is a  conjugacy class in a finite group $G$, then the rack $\Oc$ is of type D if and only if
\ref{item:typeD-group} holds, see \cite{AFGV-ampa}.

\section{Algebraic groups}\label{subsec:prelim-uno}
Let $\G$ be a  connected reductive algebraic group
defined over  $\kk = \overline{\F_q}$ and let $F: \G\to\G$ be a Frobenius map, that is 
a $\F_{q}$-split Steinberg endomorphism \cite[Chapter 21]{MT}. 
Thus there exists an $F$-stable torus $\T$ such that $F(t)=t^q$ for  $t\in \T$, 
and   $\G^F = \G(\kc)$ is the finite group of $\kc$-points. We make more precise assumptions on $\G$ in each Subsection below.
The main objectives of the paper are encompassed in the following situations:

\begin{itemize}[leftmargin=*]\renewcommand{\labelitemi}{$\diamond$}
 
\item The group $\G$ is  either $\SL_{n}(\kk)$ or $\Sp_{2n}(\kk)$ ($n\geq2$)
and  \[\Gb\coloneqq \G^F/Z(\G^F)=[\G^F,\G^F]/Z([\G^F,\G^F])\] is a finite simple  group.

\item 
The group $\G$ is  either $\SO_{2n+1}(\kk)$ ($n\geq2$, $p$ is odd) 
or $\SO_{2n}(\kk)$  ($n\geq 4$) and $\Gb\coloneqq [\G^F, \G^F]/Z(\G^F)$ is a finite simple   group.
\end{itemize}

\medbreak
In both situations we say that $\Gb$ is a (classical) Chevalley group.
We allow $\SO_{5}(\kk)$ whose simply connected cover is $\Sp_{6}(\kk)$, and $\SL_2(3)$, $\SL_2(4)$, $\SL_2(5)$, $\SL_2(9)$ for flexibility in some recursive arguments.

\medbreak
Let $\Phi$ be the root system of $\G$ and fix a subset $\Delta$ of simple roots.
Let $Q$, respectively $\Lambda$, be the root, respectively  weight, lattice.
Then $\Lambda = \oplus_{i\in \I} \Z \omega_i$; here $\theta$ is the rank of $\G$, $\I = \I_{\theta}$
and $(\omega_i)_{i\in \I}$ are the fundamental weights. Let $W$ be the Weyl group of $\Phi$
and let $s_{\alpha} \in W$ be the reflection corresponding to $\alpha \in \Phi$.
Also, $\alpha^{\vee}_i: \kk^{\times} \to \T$, $i\in \I$, are the simple coroots.
Then 
\begin{align*}
\omega_i(\alpha^{\vee}_j(\xi)) &= \xi^{\delta_{ij}}, & \xi &\in \kk^{\times}, \ i, j \in \I.
\end{align*}

\medbreak
If $\alpha\in\Phi$, then there is a
monomorphism  $x_\alpha: \kk\to \G$ of abelian groups; let
$\U_{\alpha} =\Imm x_\alpha$ (a root subgroup) and
%We recall the commutation rule:
%\begin{align}\label{eq:comm-rule}
%t\trid x_{\alpha}(a) &= t x_{\alpha}(a)t^{-1}=x_{\alpha}(\alpha(t)a),& t\in \T, \ \alpha&\in\Phi.
%\end{align}
let $\U$, respectively $\U^{-}$,  be the subgroup of $\G$ generated by the $\U_{\alpha}$'s with $\alpha\in\Delta$,
respectively $-\alpha\in\Delta$.

\subsection{The classical groups}\label{subsec:prelim-classical}
In this section we fix notation for the classical groups we will deal with. For $m\geq1$ we set ${\Jf}_m=
\left(\begin{smallmatrix}
&&1\\
&\reflectbox{$\ddots$}&\\
1&&
\end{smallmatrix}\right)$. We denote by $\Fr_q$ the Frobenius map  $\GL_{m}(\kk) \to \GL_{m}(\kk)$ given by 
$(a_{ij}) \mapsto (a_{ij}^q)$, and similarly the restriction to any suitable subgroup.

\medbreak 

We will often  use the automorphism  $\phi \colon \GL_m(\kk)\to \GL_m(\kk)$   given by: 
\begin{align}\label{eq:def-phi}
\phi(A)&\coloneqq\Jf_m \,^t\!A^{-1}\Jf_m.
\end{align}

\subsubsection*{\bf The symplectic group $\Sp_{2n}(\kk)$.}
The symplectic group $\Sp_{2n}(\kk)$ is the subgroup of $\GL_{2n}(\kk)$ leaving invariant the
bilinear form  $\left(\begin{smallmatrix}
0&{\Jf}_n\\
-{\Jf}_n&0\end{smallmatrix}\right)$.
Thus $\Sp_{2n}(\kk)$ consists of the invertible matrices $\left(\begin{smallmatrix}
A & B \\ C & D\end{smallmatrix}\right)$ such that
\begin{align}\label{eq:sympl-intro}
{}^t\hspace{-2pt}C {\Jf}_nA &={}^t\hspace{-2pt}A {\Jf}_nC,& {}^t\hspace{-2pt}B {\Jf}_nD &={}^t\hspace{-2pt} D {\Jf}_nB,& 
- {}^t\hspace{-2pt} C {\Jf}_nB +  {}^t\hspace{-2pt} A {\Jf}_nD &= {\Jf}_n.
\end{align}
 In this case, $F=\Fr_q$ and $\Gb=\Sp_{2n}(q)/Z(\Sp_{2n}(q))=\PSp_{2n}(q)$.

 \bigbreak

\subsubsection*{\bf The orthogonal group $\SO_{2n+1}(\kk)$.}
Let $p$ be odd. The orthogonal group $\SO_{2n+1}(\kk)$ is the subgroup of $\SL_{2n+1}(\kk)$ leaving invariant the
bilinear form   ${\Jf}_{2n+1}$. Thus $\SO_{2n+1}(\kk)$ consists of the invertible matrices  
\begin{align*}
X &= \left(\begin{smallmatrix}
A & e& B \\ f& k & g\\ C &h & D\end{smallmatrix}\right), & A, B, C, D &\in \kk^{n \times n},&
e, {}^t\hspace{-3pt}f, {}^tg, h &\in \kk^{n \times 1},&  k &\in \kk
\end{align*}
 such that $\det X = 1$ and
\begin{align}\label{eq:orthogonal-odd-intro}
\begin{aligned}
 {}^tC {\Jf}_nA + {}^t\hspace{-3pt}f f + {}^tA {\Jf}_nC &=0,
&
{}^tC {\Jf}_n e + {}^t\hspace{-3pt}f k + {}^tA {\Jf}_n h &=0,
\\
{}^tC {\Jf}_nB + {}^t\hspace{-3pt}f g +  {}^tA {\Jf}_nD &= {\Jf}_n,
& {}^th {\Jf}_n e + k^2 + {}^te {\Jf}_n h& =1,
\\
{}^th {\Jf}_nB + k g +  {}^te {\Jf}_nD &= 0,
& {}^tD {\Jf}_n B + {}^tgg + {}^tB {\Jf}_n D& = 0.
\end{aligned}
\end{align}
In this case $F=\Fr_q$ and $\Gb=[\SO_{2n+1}(q), \SO_{2n+1}(q)]=\Pom_{2n+1}(q)$.

\subsubsection*{\bf The orthogonal group $\SO_{2n}(\kk)$.}
Let $n\geq 4$. The orthogonal group $\SO_{2n}(\kk)$ is the subgroup of  matrices in $\SL_{2n}(\kk)$ preserving the quadratic form $\sum_{i=1}^nx_ix_{2n-i+1}$. If $p$ is odd, such elements automatically preserve the
bilinear form  with associated matrix  ${\Jf}_{2n}$. Thus $\SO_{2n}(\kk)$ consists of those  matrices $\left(\begin{smallmatrix}
A & B \\ C & D\end{smallmatrix}\right)\in \SL_{2n}(\kk)$ with $A, B, C, D \in \kk^{n \times n}$, such that
\begin{align}\label{eq:orthogonal-even-odd-intro} 
{}^tC {\Jf}_nA + {}^tA {\Jf}_nC &=0,\ 
{}^tC {\Jf}_nB +  {}^tA {\Jf}_nD = {\Jf}_n,& {}^tD {\Jf}_n B + {}^tB {\Jf}_n D& = 0.
\end{align}
If $p=2$, then $\SO_{2n}(\kk)$ consists of those  matrices $\left(\begin{smallmatrix}
A & B \\ C & D\end{smallmatrix}\right)\in \SL_{2n}(\kk)$ satisfying \eqref{eq:orthogonal-even-odd-intro} and such that 
the diagonal terms in ${}^tC {\Jf}_nA$ and   ${}^tB {\Jf}_n D$ are 0.
In this case $F=\Fr_q$ and \begin{align*}
\Gb=[\SO_{2n}(q),\SO_{2n}(q)]/Z([\SO_{2n}(q),\SO_{2n}(q)])=\Pom^+_{2n}(q).\end{align*}

\subsection{On normalizers}\label{subsec:normalizers}
In Subsection \ref{subsec:sln-prime} we shall need 
the finite unitary groups $\SU_{n}(q)$ and $\GU_{n}(q)$, and the following folklore fact. We consider:
\begin{itemize} [leftmargin=*]\renewcommand{\labelitemi}{$\circ$}
\item $\G=\SL_n(\kk)$,   $q_0=p^{m_0}$ with $m_0|m$ so that $q$ is a power of $q_0$ ;

\medbreak
\item $F_0\colon \GL_n(\kk)\to\GL_n(\kk)$ is defined either by  $F_0(A)\coloneqq \Fr_{q_0}(A)$ or by
 $F_0(A)\coloneqq \Fr_{q_0^{1/2}}(\phi(A))$,  for $A\in \GL_n(\kk)$,  $\phi$ as in \eqref{eq:def-phi},  the latter occurring only for $m_0$ even, in which case,  we denote as usual   $\SU_{n}(q_0^{1/2}) = \G^{F_0}$ and $\GU_{n}(q_0^{1/2}) = \GL_n(\kk)^{F_0}$.
\end{itemize}

\begin{prop}\label{prop:normaliserSLq0} 
$N_{\GL_n(q)}(\G^{F_0})= Z(\GL_{n}(q)) \GL_n(\kk)^{F_0}$.
\end{prop}
\begin{proof}
We prove that $N_{\GL_n(q)}(\G^{F_0})\leq  Z(\GL_{n}(q))\GL_n(\kk)^{F_0}$,  the other inclusion being immediate.
Let $g\in N_{\GL_n(q)}(\G^{F_0})$.  For any $y\in \G^{F_0}$ there holds $F_0(gyg^{-1})=gyg^{-1}$,  that is  $z\coloneqq g^{-1}F_0(g)\in C_{\GL_n(q)}(\G^{F_0})$. 
Now,  $\G^{F_0}$ contains regular unipotent elements in $\U$ and $\U^{-}$,  so  it follows  from \cite[Lemma 5.3]{springer-arithmetical} that
$C_{\GL_n(q)}(\G^{F_0})=Z(\GL_n(q))$.   In addition, $F_0$ restricts to a Steinberg endomorphism on the connected group $Z(\GL_n(\kk))\simeq \kk^{\times}$,  hence Lang-Steinberg theorem \cite[Theorem 21.7]{MT} is in force and there exists $\zeta\id\in  Z(\GL_n(\kk))$ such that $\zeta^{-1}F_0(\zeta)\id=z$.  Hence, $\zeta^{-1}g\in\GL_n(\kk)^{F_0}\leq\GL_n(q)$ 
and so $\zeta\id\in Z(\GL_n(\kk))\cap\GL_n(q)=Z(\GL_n(q))$.  The claim follows. \end{proof}

\subsection{The subgroup \texorpdfstring{$K$}{}}\label{subsec:K}
We introduce a subgroup $K$ of $\G$ that will be useful in Sections \ref{sec:split} and \ref{sec:K}.
In this Subsection $\G$ is one of the  groups
\begin{align*}
&\Sp_{2n}(\kk), \ n\geq 2; & &\SO_{2n}(\kk),\ n\geq 3,& &\text{or \qquad} \SO_{2n+1}(\kk), \ n\geq 3,
\end{align*}
where $p\neq 2$ when $\G=\SO_{2n+1}(\kk)$. We set $n'=2n$ if $\G=\SO_{2n}(\kk)$ or $\G=\Sp_{2n}(\kk)$ and 
$n'=2n+1$ if $\G=\SO_{2n+1}(\kk)$, so that $\G\leq \GL_{n'}(\kk)$. 

\medbreak

Recall $\phi$ from \eqref{eq:def-phi}. Then $K$ is the image of the injective group morphism 
\begin{align*}
j&\colon\GL_n(q)\to \G^F, &
A&\mapsto\begin{cases} \left(\begin{smallmatrix}A&0\\
0&\phi(A)\end{smallmatrix}\right), &\textrm { if }\G=\SO_{2n}(\kk), \textrm { or }\Sp_{2n}(\kk),\\
 \left(\begin{smallmatrix}A&0&0\\
0&1&0\\
0&0&\phi(A)\end{smallmatrix}\right), & \textrm{ if } \G=\SO_{2n+1}(\kk).\end{cases}
\end{align*}

\subsection{Cuspidal classes in the Weyl group}\label{subsec:prelim-unoymedio}

Let $S = \{s_{\alpha}: \alpha\in \Delta\}$, so that $(W, S)$ is a Coxeter group. Given $J \subset \Delta$, 
we set \begin{itemize} [leftmargin=*]\renewcommand{\labelitemi}{$\circ$}
\item $W_J = \langle s_{\alpha}: \alpha\in J\rangle$;  
\item $\Pa_J = $ the standard parabolic subgroup of $\G$ determined by $J$; 
\item $\Lb_J = $  the standard (reductive) Levi subgroup of $\Pa_J$.
\end{itemize}

\begin{definition} \cite[3.1.1]{GP}  A conjugacy class $\C$ in $W$ is called \emph{cuspidal}
if $\C \cap W_J = \emptyset$ for all proper subsets $J$ of $S$;  an element is cuspidal if its conjugacy class is so.
\end{definition}

A decomposition of  $w\in W$ is a family  $\varGamma = (\gamma_j)_{j\in \I_{l}}$ in $\Phi$, such that
\begin{align}\label{eq:tau-decomposition}
w &= s_{\gamma_1}\cdots s_{\gamma_l},
\end{align}
where $s_{\gamma_j}$ is the corresponding reflection and $l$ is minimal (with this property). Then 
$l$ is denoted by $\ell^a(w)$ and is called the absolute length of $w$. 
%Do not confuse with the usual length of Coxeter groups defined in terms of reflections corresponding to simple roots only.
By a result of Kostant, see \cite{PW},
$\varGamma$ is then a linearly independent family and 
\begin{align}\label{eq:rk-decomposition}
\ell^a(w) &= \operatorname{rk} (\id - w)
\end{align}
in the natural representation of $W$. By \cite[Exercise 3.16]{GP}, we have
\begin{align}\label{eq:cuspidal-decomposition}
w \text{ is cuspidal }  & \iff \ell^a(w) =  \operatorname{rk} \G.
\end{align}
Notice that  $\ell^a(w) =  \operatorname{rk} \G$ means that $w$ has no fixed points.

\medbreak
Given a decomposition $\varGamma$ of $w$, we set
\begin{align}\label{eq:root-subsystem}
\Psi_{\varGamma} &=\Phi\cap \left(\Z\gamma_1 \oplus \dots \oplus \Z\gamma_l\right),
\\ \label{eq:subgrp-decomposition}
\G_\varGamma & = \langle \T, \U_\beta: \beta\in\Psi_\varGamma\rangle.
\end{align} 
Clearly, $\Psi_{\varGamma}$ is a root subsystem of $\Phi$ and 
$\G_\varGamma$ is  a connected reductive subgroup  of $\G$.
If $\varGamma$ and $\varGamma'$ are different  decompositions of the same $w$, 
then the subsystems $\Psi_{\varGamma}$ and $\Psi_{\varGamma'}$,
and the subgroups $\G_{\varGamma}$ and $\G_{\varGamma'}$, might be different.

\begin{obs}\label{obs:decomposition-cuspidal}
If $w \in W$ is cuspidal, then $\G_\varGamma$ is semisimple for any decomposition $\varGamma$ of $w$, by
\eqref{eq:cuspidal-decomposition}. 
\end{obs}

\begin{obs}\label{obs:decomposition-levi}
If $w \in W_J$ for some $J \subset S$, then there is a decomposition $\varGamma$ 
such that $\G_\varGamma \leq \Le_J$. 
\end{obs}

Indeed, any decomposition of $w$ in $W_J$ is necessarily a decomposition in $W$, by  
\eqref{eq:rk-decomposition}. For, $w$ acts trivially in $(\mathbb R J)^{\perp}$, hence 
$\operatorname{rk} (\id - w) = \operatorname{rk} (\id - w)_{\vert \mathbb R J}$.

\subsection{\texorpdfstring{$F$}{}-stable tori}\label{subsec:prelim-dos}
Here we assume that $\G$ is connected reductive and  $F$ is a Frobenius map.
By \cite[Proposition 25.1]{MT}, there is a bijection 
from the set of $\G^F$-conjugacy classes of $F$-stable maximal tori to 
the set of conjugacy classes in $W$,  described as follows.
Let $\T'$ be an $F$-stable torus in $\G$. 
Then $\T'= g\T g^{-1}$ for some $g \in \G$ such that $g^{-1}F(g) \in N(\T)$. 
Let 
\begin{align}\label{eq:w}
w = \text{ class of }  g^{-1}F(g) \in N(\T)/ \T \simeq W.
\end{align}
The assignment $\T' \mapsto w$ gives rise to the mentioned bijection. 
We set 
\begin{align}\label{eq:Tw}
\T_w \coloneqq  g\T g^{-1}. 
\end{align}

\begin{obs}\label{rem:coxeter-reductive}
Let $\T'$ be an $F$-stable maximal torus in  $\G$ such that $\T' \mapsto w$ in the correspondence above and let $\G'$ be the derived group of $\G$.
Then $\T' \cap \G'$ is an $F$-stable maximal torus in $\G'$ and $\T' \cap \G' \mapsto w$. 
\end{obs}

Indeed, $\T \cap \G'$ is a split torus of $\G'$. The element $g \in \G$ such that $\T'= g\T g^{-1}$ and $g^{-1}F(g) \in N_{\G}(\T)$
is a representative of $w$ can be written as $g= g'z$ where $g' \in \G'$ and $z$ is central. Then
$\T' \cap \G' = g'(\T \cap \G')(g')^{-1}$ and $(g')^{-1}F(g') \in N_{\G'}(\T \cap \G')$
is a representative of $w$.

\begin{definition}
An $F$-stable maximal torus is \emph{cuspidal} if the
corresponding conjugacy class in $W$ as above is cuspidal.
\end{definition}

\begin{exa}\label{exa:coxeter} Let $w \in W$ be a Coxeter element, i.e. a product
\begin{align}\label{eq:coxeter}
w = s_1 \dots  s_\theta,
\end{align} 
where $(s_i)_{i\in \I_{\theta}}$ is a numeration of $S$; this provides a decomposition $\varGamma$ of $w$. 
Then the class of $w$ is cuspidal and $\G_{\varGamma} = \G$. If $W={\mathbb S}_n$, the conjugacy class of $w$ is the only cuspidal class in $W$, \cite[\S 3.1.2]{GP}.

\begin{definition}
A \emph{Coxeter torus} is an $F$-stable maximal torus that
corresponds to the conjugacy class containing a Coxeter element.
\end{definition}

\medbreak
By abuse of terminology the intersection of a Coxeter torus of $\G$ with $\G^F$ will be called 
a Coxeter torus of $\G^F$.
\end{exa}

\subsection{Semisimple classes}\label{subsec:prelim-tres} Here $\G$ is 
connected and reductive, unless otherwise stated,  $F$ is a Frobenius map and $\T$
is an $F$-stable torus such that $F(t)=t^q$ for  $t\in \T$.

\medbreak
Let $x \in \Gb = \G^F/Z(\G^F)$ be semisimple non-trivial and pick $\x \in \G^F$ a representative of $x$, thus $\x$ is semisimple  but not central.  Let $\oc = \oc_{x}^{\Gb}$ and $\Ot = \oc_{\x}^{\G^F}$; there is an epimorphism of racks $\Ot \twoheadrightarrow \oc$.

\medbreak

Let $\y\in \G^F$ be semisimple. By \cite[Proposition 26.6]{MT}, there exists an $F$-stable maximal torus $\T'$ containing $\y$; 
however, not all $F$-stable maximal tori intersecting $\Oc_{\y}^{\G^F}$ are necessarily $\G^F$-conjugated to 
$\T'$. Consequently we assign to $\Oc_{\y}^{\G^F}$ the set $\Ss_{\Oc_{\y}^{\G^F}}$ of all conjugacy classes $\C$ in $W$ corresponding to
 $F$-stable maximal tori $\T'$ that intersect $\Oc_{\y}^{\G^F}$. 
 
 \begin{obs}\label{obs:SO-derived}
Assume that $\G$ is simple and that we are not in the cases excluded in \cite[Theorem 24.17]{MT}. If $\C\in \Ss_{\Oc_{\y}^{\G^F}}$,  then $\Oc_{\y}^{[\G^F,\G^F]}$ intersects an $F$-stable maximal torus $\T''$ in $\G^F$ corresponding to an element in $\C$. 
Indeed, let $g\in \G^F$ be such that $g^{-1}F(g)\in N_{\G}(\T)$ represents an element in $\C$, and let $\T'=g\T g^{-1}$.  Then there exists $l\in\G^F$ such that $l\trid \y\in\T'$.
 Now, $\G^F=\T^F[\G^F,\G^F]$, \cite[Corollary 24.2, Proposition 24.15, Proposition 24.21]{MT}, so $l$ decomposes as $l=t_1l_1$ with $t_1\in\T^F$ and  $l_1\in [\G^F,\G^F]$, and $l_1\trid\y\in t_1^{-1}g\T g^{-1}t_1=\T''$, where $g^{-1}t_1F(t_1^{-1})F(g)=g^{-1}F(g)$ represents an element in $\C$.
 \end{obs}
 
For our aim, it is convenient to introduce the following notion.

\begin{definition}  A semisimple conjugacy class $\Oc_\y^{\G^F}$ in $\G^F$ is called \emph{cuspidal}
if the set $\Ss_{\Oc_{\y}^{\G^F}}$ consists of cuspidal conjugacy classes in $W$.
In other words, all  $F$-stable maximal tori intersecting $\Oc_\y^{\G^F}$ are cuspidal.

\smallbreak
Also, $\Oc_\y^{\G^F}$ is called a \emph{Coxeter class}
if it only intersects Coxeter tori. Necessarily, $\Oc_\y^{\G^F}$ is then cuspidal.
\end{definition}

\begin{obs}
Since $Z(\G^F)=Z(\G)^F$ is contained in every torus of $\G^F$,  the class ${\Oc_\y^{\G^F}}$ is cuspidal,  respectively Coxeter,  if and only if its projection $\Oc'$  in  $\G^F/Z(\G^F)$
intersects only cuspidal tori,  respectively Coxeter tori in $\G^F/Z(\G^F)$. We will thus call also $\Oc'$ cuspidal, respectively Coxeter. In particular, if  $\G$ is simply-connected and $\Ot$ is cuspidal,  $\Oc$ will be called cuspidal.  

If $\G$ is not simply-connected,  Remark \ref{obs:SO-derived} guarantees that ${\Oc_\y^{\G^F}}$ is cuspidal, respectively Coxeter, if an only if $\Oc_{\y}^{[\G^F,\G^F]}$ intersects only cuspidal tori,  respectively Coxeter tori in $\G^F$, and we will call also $\Oc_{\y}^{[\G^F,\G^F]}$ cuspidal, respectively Coxeter.  
\end{obs}

\begin{prop}\label{prop:cuspidal-classes} Assume $\G$ is simply-connected.  If $\T'$ is a maximal $F$-stable torus that intersects $\Ot$, 
$\C$ is the conjugacy class in $W$ corresponding to $\T'$ 
and $\varGamma$ is a decomposition of $w \in \C$, then $\Ot$ intersects $\G_{\varGamma}^F$. 
In particular, the following are equivalent:
\begin{enumerate}[leftmargin=*,label=\rm{(\alph*)}]
\item $\Ot$ is not cuspidal.
\item $\Ot$ intersects 
a proper standard Levi subgroup $\Lb$.
\end{enumerate}
\end{prop}

\pf Since $F$ is a Frobenius automorphism, $\G_{\varGamma}$ is $F$-stable. Pick a representative $\dot{w}$ of $w$; by definition, it belongs to $\G_{\varGamma}$.
By the Lang-Steinberg Theorem there is $h\in\G_{\varGamma}$ such that $h^{-1}F(h)=\wdot$. 
The tori $\T'$ and $h\T h^{-1}$ are $\G^F$-conjugate since they both map to $w$, 
cf. \eqref{eq:w}. 
That is, there exists $\y\in\G^F$ such that $\y\trid \x\in h\T h^{-1} \leq \G_{\varGamma}$,   hence $\y$ decomposes as $\y=\tt\y'$ for some $\y\trid \x\in \G_{\varGamma}^F \cap \Ot$.

If $\Ot$ is not cuspidal, then pick $\C$ non-cuspidal and apply Remark \ref{obs:decomposition-levi}.
Conversely, if  $\y \in \Ot \cap \Lb$, then there is an $F$-stable maximal torus $\T'$ of $\Lb$ that contains $\y$.
Hence $\T' = u \T u ^{-1}$ for some $u \in \Lb$ such that $\sigma \coloneqq u ^{-1}F(u) \in N_{\Lb}(\T) \leq N_{\G}(\T)$; so that the class of $\sigma$ belongs to the Weyl group of $\Lb$.
\epf

\bigskip

Recall that a semisimple element $y$ is \emph{regular}
if its centraliser $C_{\G}(y)$ consists of semisimple elements, or equivalently, if the irreducible component $C_{\G}(y)^\circ$ of $C_{\G}(y)$  containing the identity is a torus, \cite[II.11]{steinberg-regular}. This occurs if and only if $y$ lives in a unique maximal torus.
If our $\x \in \G^F$ is regular, then $C_{\G}(\x)^\circ$ is the unique $F$-stable maximal torus containing $\x$. 

\begin{prop}\label{prop:one-conj-class-regular-bis}
If $\x$ is a cuspidal element, then it is regular.
\end{prop}

We thank Gunter Malle for suggesting us the following proof. 

\pf Let $\T_0$ be an  $F$-stable maximal torus of $\G$ containing $\x$, 
let $g\in\G$  be such that $\T_0=g\T g^{-1}$ and let $w$  be the corresponding  Weyl group element, i.e., $g^{-1}F(g)\in w\T$
as in \eqref{eq:w}. 
The $F$-stable maximal tori in $\G$ containing $\x$ are also the $F$-stable maximal tori in  the connected reductive group $C=C_{\G}(\x)^\circ$.  
Every $F$-stable maximal torus in $C$ is of the form $c\T_0c^{-1}$ for  some $c\in C$ such that $c^{-1}F(c)\in N_C(\T_0)\leq gN_{\G}(\T)g^{-1}$.
Let $W_C=N_C(\T_0)/\T_0$ be the Weyl group of $C$. We claim that $W_C$ is trivial.  Assume for a contradiction that $W_C$ is non-trivial. Let $s$ be a reflection in $W_C$ and let $c\in C$ be such that $c^{-1}F(c)=\dot{s}$, a representative of $s$  in $N_C(\T_0)$. Then, $\dot{s'}\coloneqq g^{-1}\dot{s} g$ would represent a reflection $s'$ in $W$ and $c\T_{0}c^{-1} = cg\T g^{-1}c^{-1}$ is an $F$-stable maximal torus of $\G$, containing $\x$ and corresponding to $g^{-1}c^{-1}F(c)F(g)=(g^{-1}c^{-1}F(c)g)(g^{-1}F(g))\in s'w\T$. Therefore, $s'w$ is cuspidal by hypothesis on $\x$. 
However, the characteristic polynomial of a cuspidal element is a product of cyclotomic polynomials different from $(X-1)$, therefore its value at $0$ is $1$, see \eqref{eq:rk-decomposition} and \eqref{eq:cuspidal-decomposition}. On the other hand,  $\det(s'w)=-\det (w)$. Hence, $s'w$ and $w$ can not be both cuspidal elements in $W$, contradicting our assumption on $\x$. Therefore $W_C$ has no reflections and $C=\T_0$ is the unique maximal  torus containing $\x$.
\epf

\medbreak
The following well-known result is instrumental to apply Lemma \ref{lem:typeC-directproduct}. 

\begin{lema}\label{lema:x-xq}Assume $\G$ is  simply-connected.\\
\begin{enumerate*}[leftmargin=*,label=\rm{(\alph*)}]
\item\label{item:x-xq1} $\x^q \in \Ot$.

\item\label{item:x-xq2} If $\x^q = \x$, then $\Ot \cap \T^F \neq \emptyset$.
\end{enumerate*}
\end{lema}

\pf First, $\oc_{\x}^{\G}$ is $F$-stable: if $\x = h\y h^{-1}$ for some $h\in \G$, then 
\[F(\y) = F(h)^{-1} \x F(h) \in \oc_{\x}^{\G}. \] 
Since $\x$ is semisimple, there are $t \in \T$ and $g \in \G$ such that $\x = gtg^{-1}$. 
Thus $t^q = F(t) \in \oc_{\x}^{\G}$ and consequently $\x^q \in \oc_{\x}^{\G} \cap \G^F = \Ot$;
here the last equality holds because $C_{\G}(\x)$ is connected, 
$\G$ being simply connected, cf. \cite[\S 2.11, \S 8.5]{hu-cc}.  
Finally, if $\x^q = \x$, then $t^q = t \in \T^F \cap  \oc_{\x}^{\G} \subset \Ot$
by the same reason. \epf

\section{Split conjugacy classes}\label{sec:split}

We keep the notation from  \S \ref{subsec:prelim-tres}, namely 
$\G$ is simple and simply connected, but not of type A.
Also $F$ is a Frobenius map; $\T$
is an $F$-stable torus such that $F(t)=t^q$ for  $t\in \T$;
$e \neq x \in \Gb = \G^F/Z(\G^F)$ is semisimple; 
$\x \in \G^F$ a representative of $x$; $\oc = \oc_{x}^{\Gb}$ and $\Ot = \oc_{\x}^{\G^F}$.
Thus there is an epimorphism of racks $\Ot \twoheadrightarrow \oc$.

\medbreak
We assume additionally  that $\Ot\cap\T^F\neq\emptyset$.
Without loss of generality, we suppose that $\x \in \T^F$,  i.e., $\x$ is split. 
Adapting the proof of \cite[Lemma 3.9]{ACG-III} for type $A$, but with more work, 
we deal with such classes. 

\medbreak
We will need to consider  separately the following particular situation:
\begin{align}\label{eq:special-situation}
&\G \text{ is of type }B_\theta, \ q \text{ is odd, } \x \text{ satisfies } s_{\alpha_j}(\x) =
\begin{cases}
 \x &\text{if } j< \theta, \\ \alpha^{\vee}_{\theta}(-1)\x &\text{if } j = \theta.
\end{cases}
\end{align}
Here  $\theta \geq 2$ (as $B_2 = C_2$). 
When this is the case, then $\x$ has the form
\begin{align}\label{eq:special-situation-explicit}
 & \x= \Bigg(\prod_{i\in \I_{\theta -1}} \alpha^{\vee}_i ((-1)^{i})\Bigg) \alpha^{\vee}_{\theta} (\eta),&
\text{where }  \eta^2&= (-1)^{\theta}.
\end{align}
Notice that if $\theta$ is odd, then such a $\x$ belongs to $\G^F$ iff $q \equiv 1 \mod 4$.

Here is the main result of this Section:

\begin{theorem}\label{thm:split-collapses} Assume that $q > 2$; $\G$ is not of type $A_{\theta}$;
$q\notin\{3,5,7\}$ if we are in \eqref{eq:special-situation}; and 
$\Oc_{\x}^{\G^F}$ intersects the split torus $\T^F$. Then $\oc$ collapses. 
\end{theorem}
When $q = 2$,  $\T^F$ is trivial and the class of $\x$ could not intersect it.

\subsection{Proof of Theorem \ref{thm:split-collapses}}
This follows from Lemmata \ref{lem:not-special} and \ref{lem:special-q-big}.

\begin{lema}\label{lem:not-special} Assume that $q>2$ and that we are \textbf{not} in the situation \eqref{eq:special-situation}.  
Then $\Oc$ is of type C.
\end{lema}

\begin{table}[ht]
\caption{Center of some $\G$, $q$ odd; $\zeta \in \F_q^{\times}$ has order 4}\label{tab:center}
\begin{tabular}{|c|c|c|}
\hline  type & q &  $Z(\G)$    \\
\hline
$B_\theta$ & & $\langle \alpha^{\vee}_\theta(-1)\rangle$ 
\\
\hline
  $C_\theta$, $\theta > 2$ & & $\left\langle \displaystyle\prod_{i\text{ odd}}\alpha^{\vee}_i(-1)\right\rangle$
\\
\hline
$D_{\theta}$, & $q \equiv 1 \mod 4$
 & $\left\langle \displaystyle \prod_{i\text{ odd}, i\leq \theta - 2}\alpha^{\vee}_i (-1)\alpha^{\vee}_{\theta -1}(\zeta)\alpha^{\vee}_{\theta}(\zeta^3)\right\rangle$
\\
\cline{2-3}
 $\theta\in 2\Z + 1$ & $q \equiv 3 \mod 4$
& $\left\langle \displaystyle \alpha^{\vee}_{\theta -1}(-1)\alpha^{\vee}_{\theta}(-1)\right\rangle$
\\
\hline
$D_{\theta}$,  $\theta \in 2\Z$  & 
& $\left\langle \displaystyle\prod_{i \text{ odd}} \alpha^{\vee}_i(-1),\,
\alpha^{\vee}_{\theta-1}(-1)\alpha^{\vee}_{\theta}(-1)\right\rangle$\\
\hline
$E_7$ & &$\langle  \alpha^{\vee}_2(-1)\alpha^{\vee}_5(-1)\alpha^{\vee}_7(-1) \rangle$
\\
\hline
\end{tabular}
\end{table}

\begin{proof}Recall that $\x\in\T^F$.  We will rely on the proof of \cite[Lemmata 4.1, 4.2]{ACG-V}. It is shown there that, for any simple root $\alpha$ such that $s_\alpha(\x)\neq\x$,  the subrack  $\Y=\x \U_\alpha^F\coprod s_\alpha(\x)\U_\alpha^F$ 
of $ \Oc_{\x}^{\G^F}$ is  of type C. We claim that we can choose $\alpha$ such that the restriction of the projection $\pi: \Y \to \Oc$ is injective.
If $\G$ is of type $E_8, F_4, G_2$, then $Z(\G)$ is trivial and $\Gb = \G^F$.
Let $u \neq v \in \Y$ such that $\pi(u) = \pi(v)$, i.e. there is $z \in Z(\G^F)$, $ z\neq 1$, such
$u = zv$. Hence either $u = \x x_{\alpha}(a)$ and $v = s_\alpha(\x) x_{\alpha}(a)$, or vice versa.
In any case, $\x \overset{\star}{=} zs_\alpha(\x)$. Applying $s_\alpha$, we get $z^2 = 1$. Thus $\G$ is not of type $E_6$ 
(here $Z(\G) \simeq \Z/3$); and $q$ should be odd. By $\star$, we have
\begin{align}
\omega_i(\x) &= \omega_i(zs_\alpha(\x)) = \omega_i(z) s_\alpha(\omega_i)(\x) , & i &\in \I_\theta.
\end{align}
Say $\alpha = \alpha_j$, $j \in \I_\theta$. Then  $s_\alpha(\omega_i) = \omega_i$ when $i \neq j$, hence 
$\omega_i(z) = 1$. Now such $z$ exists only  in the situation \eqref{eq:special-situation}, see
the shape of $Z(\G)$ in Table \ref{tab:center}. \end{proof}

\begin{lema}\label{lem:special-q-big}
 If we  \textbf{are} in situation \eqref{eq:special-situation} with $q = 9$,  
then  $\Oc$ is of type C. If we  \textbf{are} in situation \eqref{eq:special-situation} with $q > 9$,  
then  $\Oc$ is of type D.\end{lema}
\begin{proof}
We deal first with $\theta = 2$; now $B_2 = C_2$ and $\G  = \Sp_4(\kk)$. 
Let $K\simeq \GL_2(q)$ be the subgroup of $\G^F$ which is the image of $j$  as in \S \ref{subsec:K}.
Thus we have a monomorphism of groups  
$\GL_2(q)/\{\pm 1\} \to \Gb = \PSp_4(q)$.

 Let   $\y =  \left(\begin{smallmatrix}1& 0 \\ 0 &-1\end{smallmatrix}\right) 
 \in \GL_2(q)$; by our assumption on $\x$, we know  that either $\x \overset{\star}{=} j(\y) $ or $\x = -j(\y)$.
Let $\varpi: \GL_2(q) \to \PGL_2(q)$ be the canonical projection and $y = \varpi(\y)$. Then we have a surjective map of racks
\begin{align*}
\Oc \cap K/\{\pm 1\} \to \Oc_{y}^{\PGL_2(q)}.
\end{align*} 
Therefore it is enough to prove that  $\Oc_{y}^{\PGL_2(q)}$ is of type C if $q=9$ and of type D for $q>9$.

Let $q=9$. Then $\PGL_2(9)\simeq{\mathbb A}_6$, and through this isomorphism the class $\Oc_{y}^{\PGL_2(q)}$ corresponds to the class labeled by $(1^2,2^2)$, which is of type C by Example \ref{exa:1122}.

Let now $q>9$. If $q \equiv 3 \mod 4$, then \cite[Corollary 5.4 (b)]{AFGV-simple}
applies\footnote{Notice that Corollary 5.4 (b) in \emph{loc. cit.} refers implicitly to the class of involutions in $\PGL_2(q)$ not in $\PSL_2(q)$, as is transparent from the proof. }.
Assume then that $q \equiv 1 \mod 4$. Let $\zeta \in \F_q^{\times}$ be a primitive 4-th root of 1 and
let $\tu \in \PGL_2(q)$ be the class of  
$\left(\begin{smallmatrix} \zeta& 0 \\ 0 & -\zeta\end{smallmatrix}\right)$. Then
\begin{align*}
\Oc_{y}^{\PGL_2(q)} = \Oc_{\tu}^{\PGL_2(q)}  = \Oc_{\tu}^{\PSL_2(q)} 
\end{align*}
which is of type D by \cite[Corollary 5.4 (a)]{AFGV-simple}.

\medbreak
Assume next that $\theta > 2$. Here  
\begin{align*}
\Gb = \Pom_{2\theta + 1}(q) = \G^F / Z(\G^F) \simeq [\SO_{2\theta +1}(q), \SO_{2\theta +1}(q)].
\end{align*}
We identify $\Pom_{5}(q)$ with a subgroup of $\Pom_{2\theta + 1}(q)$ via the inclusion
\begin{align*}
\SO_{5}(q) &\hookrightarrow \SO_{2\theta + 1}(q), &
\begin{pmatrix}
A & e & B \\ f& k & g \\ C& h & D  
\end{pmatrix}   &\mapsto  
\begin{pmatrix}
A & 0 & e  & 0 & B
\\ 0& \id_{\theta - 2} &0 & 0 &0 
\\ f&0& k & 0&g 
\\ 0& 0 & 0& \id_{\theta - 2}  &0
\\ C&0& h&0&D 
\end{pmatrix}, 
\end{align*}
$k\in \F_q$, $A, B, C, D \in \F_q^{2\times 2}$, etc.
Fix $t_{\theta} \in \T$ of the shape \eqref{eq:special-situation-explicit} and analogously
$t_2$ of the shape \eqref{eq:special-situation-explicit} but for type $B_2$.  
If $\pi: \G^F \to \Gb$ is the projection, then
\begin{align*}
&&\pi(t_{\theta}) &=  \diag \left( -\id_{\theta}, 1 , -\id_{\theta}  \right)= \pi(t_2) \gamma,
\\
&\text{where} &
\gamma &= \diag \left(\id_2, -\id_{\theta - 2}, 1, -\id_{\theta - 2}, \id_2 \right).
\end{align*}
Here $\diag$ refers to a diagonal of blocks. Then
\begin{align*}
\Oc = \Oc_{\pi(t_{\theta})}^{\Gb} \geq \Oc_{\pi(t_{\theta})}^{\Pom_{5}(q) \times \langle \gamma\rangle} \simeq
\Oc_{\pi(t_{2})}^{\Pom_{5}(q)} 
\end{align*}
which is of type D by the preceding argument. Hence $\Oc$ is of type D.
\end{proof}

\begin{lema}\label{lem:special-q-3}
If we  \textbf{are} in situation \eqref{eq:special-situation} with $n = 2$ and $q = 3$,  
then  $\Oc$ is austere, hence kthulhu.
\end{lema}
\begin{proof}  Indeed, $\PSp_4(3) \simeq  \PSU_4(2)$ and the semisimple class we are dealing with 
in the former group corresponds to
the unipotent class of type $(2, 1^2)$ in the latter one, which is austere by \cite[Lemma 5.2]{ACG-IV}. 
\end{proof}

\begin{obs}Assume that $n = 2$. If we  \textbf{are} in the situation \eqref{eq:special-situation} with $q =5$, 
then calculations with GAP show that $\Oc$ is austere, hence kthulhu. 
The evidence obtained by performing different computations seems
to indicate that in the case $q=7$, the class is also \textit{kthulhu}.
\end{obs}

\subsection{Split classes in orthogonal groups}
Theorem \ref{thm:split-collapses} was proved by assuming that $\G$ is simply connected.
For recursive arguments  on the orthogonal groups we need  an analogous statement  for the orbits of split elements in 
$\G^F=\SO_{n'}(q)$ for the action of $[\G^F,\G^F]$ for $n' = 2n$ or $2n+1$. 
Let $\T \leq \G$ be the subgroup of diagonal matrices
\begin{align*}
&\diag 
(t_1,\,\ldots,\,t_n,t_n^{-1},\ldots,\,t_1^{-1}),
&\textrm { if }n' &=2n,\\
&\diag 
(t_1,\,\ldots,\,t_n,1,\,t_n^{-1},\ldots,\,t_1^{-1}),
&\textrm{ if }n' &=2n+1.
\end{align*}

Recall Remark \ref{obs:normal-subrack}.

\begin{lema}\label{lema:SO-split}
If  $y \in \T^F  -  Z(\SO_{n'}(q))$, 
then $\Oc^{[\SO_{n'}(q),\SO_{n'}(q)]}_y$ collapses, 
except in the situation    \eqref{eq:special-situation}, i.e., when $n'=2n+1$ and $t_i=-1$ for  $i \in \I_n$. 
\end{lema}

\pf There is always  a simple root $\alpha$ so that the proof of \cite[Lemma 4.2]{ACG-V} carries over. 
If $t_i\neq t_{i+1}$ for some $i< n$, then take $\alpha=\alpha_i=\varepsilon_i-\varepsilon_{i+1}$; if,
instead, $t_i=t_{i+1}$ for all $i<n$, then our assumptions imply $t_{n}\neq t_n^{-1}$ and we  take $\alpha=\alpha_n$, 
i.e., $\varepsilon_n$ when $n'=2n+1$ and $\varepsilon_{n-1}+\varepsilon_n$ when $n'=2n$.
The argument works also in types $B_2$ and $D_3$.
\epf

\section{The special linear groups}\label{sec:cuspidal-sln}

In this Section $\G = \SL_{n}(\kk)$, that is,  we deal with semisimple classes in $\Gb = \PSL_{n}(q)$.
As in \S \ref{subsec:prelim-tres},
$e \neq x \in \Gb$ is semisimple, $\x \in \G^F - Z(\G^F)$ is a representative of $x$, $\oc = \oc_{x}^{\Gb}$ and $\Ot = \oc_{\x}^{\G^F}$.
There is an epimorphism of racks $\Ot \twoheadrightarrow \oc$. 

For inductive arguments, we will also consider  classes of elements in $\GL_n(q)$. As observed in \cite[Remark 4.1]{ACG-I}, for any semisimple 
element $\y\in\GL_n(q)$, we have  $\Oc_{\y}^{\GL_n(q)}=\Oc_{\y}^{\SL_n(q)}$. 

\begin{definition}\label{def:irreducible} We say that $A\in\GL_n(q)$ is \emph{irreducible}
if  its characteristic polynomial $p_A$ is irreducible; necessarily $A$ is regular semisimple.
\end{definition}

From our previous work, we know:
\begin{rmk}\label{rmks:sln-previous} \begin{enumerate*}[leftmargin=*,label=\rm{(\roman*)}]
\item\label{item:n2-qchico} \cite[Theorem 1.1]{ACG-III} If $n=2$, and $q\not\in\{2,3,4,5,9\}$, then any  $\Oc$ not listed in 
Table \ref{tab:ss-psl2}  collapses. 
\end{enumerate*} 
\begin{enumerate}[leftmargin=*,label=\rm{(\roman*)}] \setcounter{enumi}{1}
\item \cite[Props. 5.4, 5.5]{ACG-III} If $n=3$ and $\x$ is irreducible, then $\Oc$ is kthulhu.
\item \cite[Theorem 1.1]{ACG-III} If $n\geq3$ and $\x$ is not irreducible, then $\Oc$ collapses.
\end{enumerate}
\end{rmk}

\begin{obs}\label{rmks:sl2-ss}
Let $n = 2$. We record information on the  semisimple classes with  $q \in \{ 2,3,4,5,9\}$ 
for recursive arguments. Recall that $\PSL_2(q)$ has two conjugacy classes of maximal tori: the split one, of order $q-1/(2,q-1)$ and the Coxeter torus, of order $q+1/(2,q+1)$, that contains the irreducible elements.
\begin{itemize}[leftmargin=*]\renewcommand{\labelitemi}{$\circ$}
\item 
If $q = 2$, then $\PSL_{2}(2) \simeq \mathbb{S}_{3}$ and the semisimple elements are the 3-cycles that form an abelian rack; if $q=3$, then 
$\PSL_{2}(3) \simeq \mathbb{A}_{4}$  and the semisimple elements have order 2 and form an abelian rack.

\medbreak
\item 
If  $q= 4$, then 
$\PSL_{2}(4) \simeq \mathbb{A}_{5}$. The  irreducible elements have 
order 5 and form two conjugacy classes that are sober  by \cite[Remark 3.2 (b) and (c)]{F}.
The split semisimple elements  form the conjugacy class of $3$-cycles which is of type C by Example \ref{exa:3-cycles}.
\medbreak
\item 
If  $q= 5$, then  $\PSL_{2}(5) \simeq \aco$.  
The irreducible elements  form the conjugacy class  of $3$-cycles which is of type C by Example \ref{exa:3-cycles}. 
The split semisimple elements are the involutions in the class $(1,2^2)$ which is sober because its intersection with any subgroup of $\aco$ 
is either trivial, abelian or indecomposable.  
\medbreak
\item 
If  $q=  9$, then $\PSL_{2}(9) \simeq \mathbb{A}_{6}$. The  irreducible elements have order $5$ and form two conjugacy classes that are sober by \cite[Remark 3.2 (b) and (c)]{F}. 
The split semisimple elements are the involutions in the class $(1^2,2^2)$ which is of type C by Example \ref{exa:1122}.

\end{itemize}
\end{obs}

\bigbreak

Our main result in this Section is:

\begin{theorem}\label{thm:section-SL-summary}
Let $\Oc \neq \{e\}$  be a semisimple conjugacy class in $\PSL_{n}(q)$. Then any  $\Oc$ not listed in Table \ref{tab:ss-psl2}  collapses.
\end{theorem}

By Remark \ref{rmks:sln-previous}, we will consider conjugacy classes of irreducible  elements  
assuming $n > 3$. We will see that
if $n$ is  prime, then  such classes are kthulhu  by Proposition \ref{prop:irreducible-kthulhu}, otherwise, they are of type C by Proposition \ref{prop:sln-not-prime}.

\medbreak
We start by a classical result whose  proof we include for completeness.
\begin{lema}\label{lem:pol-irr}
Let $n\geq 2$ and $\epsilon=\pm1$. If $P(X)=X^n+\epsilon\in\F_p[X]$   is irreducible over $\F_q$,
then $n=2$, $\epsilon=1$ and $q\equiv3\mod4$. 
\end{lema}

\begin{proof}
First, $\epsilon=1$ and $q$ is odd, otherwise $P(1)=0$; 
and $n=2m$ is even, otherwise $P(-1)=0$. 
Also, $q\equiv 3\mod 4$, otherwise $-1=\xi^2$ for some $\xi\in \F_q$ and 
$P(X)=(X^m+\xi)(X^m-\xi)$ would be reducible.  
Let now $n=2^ha$ where $a$ is odd, and let $\Phi_d(X)$ be the $d$-th cyclotomic polynomial. We have
the factorization over ${\mathbb Z}$, hence over $\F_p$,
\begin{align*}
(X^n-1)P(X) = X^{2n}-1= \prod_{d \vert 2n} \Phi_d(X) \implies
P(X)= \prod_{d \vert 2n,\ d\nmid n} \Phi_d(X)
\end{align*}
Thus $\Phi_{2^{h+1}}\vert P(X)$, hence they are equal and $n = 2^h$.
Finally, if $X^{2^h}+1$ is irreducible over $\F_q$ for $q\equiv 3\mod 4$, then $h=1$ by \cite[Theorem 1]{Meyn}.
\end{proof}

\medbreak
\subsection{Coxeter tori}
We assume in the rest of this Section
that $n > 3$ and that $\x$ is \emph{irreducible}.
In this case  $W = \sn$ and  by Proposition \ref{prop:cuspidal-classes} and Example \ref{exa:coxeter} every irreducible class  intersects every  Coxeter torus.
We fix the $n$-cycle 
\[w = (1,2,\dots, n).\] 
By technical reasons, we fix a Coxeter torus $\T_w$  in $\GL_{n}(\kk)$; then $\T_w\cap \SL_{n}(\kk)$ is a 
Coxeter torus  in $\SL_{n}(\kk)$ by  Remark \ref{rem:coxeter-reductive}. 
By \cite[Example 25.4]{MT}, we have
\begin{align}\label{eq:gln-coxeter-torus}
| \T_w^F | &= q^n - 1 = (n)_q (q-1),& | \T_w^F\cap \SL_{n}(q)| &=  (n)_q.
\end{align}

The group $\T_w^F$ is isomorphic to $\F_{q^n}^{\times}$, hence it is cyclic; further,
any cyclic subgroup of $\GL_{n}(q)$ of order $q^n-1$ is conjugated to $\T_w^F$ \cite[Example 1.13]{Hiss}.

\medbreak
\begin{obs}\label{rem:huppert-short-sigma} (\cite[Satz II.7.3]{Huppert}  and \cite[Theorem 2.3.5 and below]{short}). We have 
\begin{align*}
N_{\GL_{n}(q)}(\T_w) &= N_{\GL_{n}(q)}(\T_w^F)\simeq \T_w^F\rtimes C_W(w), 
\\ 
N_{\SL_{n}(q)}(\T_w^F \cap \SL_{n}(q)) &= N_{\GL_{n}(q)}(\T_w^F) \cap \SL_{n}(q),
\end{align*}
with $C_W(w)\simeq  {\mathbb Z}/n$. 
Let $\sigma$ be a generator of $C_W(w)$ identified as a subgroup of  $N_{\GL_{n}(q)}(\T_w^F)$; 
$\sigma$ can be chosen  so that $\sigma\trid \y=\y^q$ for any $\y\in \T_w^F$. 
\end{obs}

\begin{lema} \label{lem:intersection}
Let $\y\in\T_w^F$ be irreducible in $\GL_n(q)$.  Then 
\begin{align}\label{eq:intersection1}
\oc_{\y}^{\GL_n(q)} \cap \T_w^F &= \Oc_{\y}^{N_{\GL_n(q)}(\T_w^F)}=\langle\sigma\rangle\trid \y= \{\y, \y^q, \dots, \y^{q^j}, \dots, \y^{q^{n-1}} \},
\\ \label{eq:intersection2}
\Oc_{\y}^{N_{\GL_n(q)}(\T_w^F)} &=\oc_{\y}^{\SL_n(q)} \cap \T_w^F =\Oc_{\y}^{N_{\SL_n(q)}(\T_w^F)}
\end{align}
\end{lema}

\pf
If $\z \in\T_w^F \cap \Oc_{\y}^{\GL_n(q)}$, then  there is $g\in \GL_n(q)$ 
such that $g\y g^{-1}=\z$, so $g C_{\GL_n(q)}(\y)g^{-1}= C_{\GL_n(q)}(\z)$, that is, $g\in N_{\GL_n(q)}(\T_w)$
since clearly $\z$ is also irreducible. This and Remark \ref{rem:huppert-short-sigma}
imply \eqref{eq:intersection1}. Since $\oc_{\y}^{\SL_n(q)} =\oc_{\y}^{\GL_n(q)}$, the centraliser argument as above gives \eqref{eq:intersection2}.
\epf

We investigate when two elements in an irreducible  class have the same image through the natural projection $\pi\colon \SL_n(q)\to\PSL_n(q)$. 
Recall that $\G_n(\F_q )$ is the group of $n$-th roots of unity in $\kc$.

\begin{lema}\label{lema:proj-irr} (Just for this Lemma, $n \geq 3$).
Let $\y$, $\z\in\Ot$ such that $\pi(\y)=\pi(\z)$, i.e., 
$\y=\lambda\z$ for some  $1 \neq\lambda  \in \G_n(\F_q )$. Then

\begin{enumerate*}[leftmargin=*,label=\rm{(\roman*)}]
\item\label{item:proj-irr1} There exists $j\in\I_{n-1}$ such that $\lambda\z=\z^{q^j}$.
\end{enumerate*}

\medbreak
\begin{enumerate}[leftmargin=*,label=\rm{(\roman*)}]\setcounter{enumi}{1}
\item\label{item:proj-irr2} Let $j\in\I_{n-1}$ be minimal satisfying $\z^{q^j}=\lambda \z$ for some  
$1 \neq\lambda  \in \G_n(\F_q )$. Then $j|n$ and $\lambda$ is a primitive $\frac{n}{j}$-th root of $1$.
\end{enumerate}

\medbreak
\begin{enumerate}[leftmargin=*,label=\rm{(\roman*)}]\setcounter{enumi}{2}
\item\label{item:proj-irr3} Let $j\in \I_{n-1}$ be minimal satisfying $\z^{q^j}=\lambda \z$ for some  
$1 \neq\lambda  \in \G_n(\F_q )$ and let $a\coloneqq \frac{n}{j}$. Then  the characteristic  polynomial $p_{\z}\in \F_q[X^a]$.
This  observation rectifies  \cite[Remark 3.1 (d)]{ACG-III}.

\medbreak
\item\label{item:proj-irr4}  Let $j \in \I_{n}$ be minimal satisfying $\pi\big(\z^{q^j}\big) = \pi(\z)$. Then $j\neq 1$. 
\end{enumerate}

\end{lema}

\pf \ref{item:proj-irr1}: The elements $\y$ and $\z$ lie in the same (unique) maximal torus, so $\y=\z^{q^j}$ for some $j\in\I_{1,n}$ by 
Lemma  \ref{lem:intersection}. Therefore, $\lambda\z=\z^{q^j}$ and $\lambda\neq 1$ implies $j<n$.

\ref{item:proj-irr2}:
If $n=aj+b$, with $a\geq 1$ and  $0 \leq b<j$, then 
\begin{align*}
\z=\z^{q^n}=\z^{q^{aj+b}}=(\z^{q^{aj}})^{q^b}=(\lambda^a\z)^{q^b}=\lambda^a\z^{q^b}
\end{align*}
that is, $\z^{q^b}=\lambda^{-a}\x$. Hence $b=0$ by minimality and $\lambda^a=1$. 
Now, if   $\lambda^c=1$ with $c\in \N$, then $\z^{q^{cj}}=\lambda^c\z=\z$, hence $c \geq a = \frac{n}{j}$.

\medbreak
\ref{item:proj-irr3}:
By assumption $p_{\z}  = p_{\z^{q^j}} = p_{\lambda  \z}$. 
If $p_{\z}(X)=X^n+c_{n-1}X^{n-1}+\cdots+c_0$,
then $p_{\lambda\z}(X)=X^n+\lambda c_{n-1}X^{n-1}+\cdots+\lambda^{n}c_0$. 
Thus $p_{\z}(X)  = p_{\lambda  \z}(X)$
if and only if $c_h=0$ for all $h\not\in \frac{n}{j}{\mathbb Z}$. 

\medbreak
\ref{item:proj-irr4}:
If $j=1$ then    $p_{\z}$ would be $X^n+(-1)^n$  by \ref{item:proj-irr3}. 
By Lemma \ref{lem:pol-irr},  $n=2$, a contradiction.
\epf

\subsection{Irreducible elements of \texorpdfstring{$\SL_n(q)$}{}, \texorpdfstring{$n$}{} not a prime}\label{subsec:sln-not-prime}

In this Subsection we assume that $n = cd$, for some $c, d \in \N_{\ge 2}$. Given
$\St \in \SL_d(q)$ irreducible, we consider $\y = \diag(\St,\,\ldots,\,\St) \in \SL_n(q)$.
Then $C_{\GL_n(q)}(\y) \simeq \GL_c(q^d)$. 

\medbreak
We claim that a Coxeter torus $\widetilde{T}$ of $C_{\GL_n(q)}(\y)$ remains a Coxeter torus in $\GL_n(q)$
hence $T \coloneqq \widetilde{T} \cap \SL_n(q)$ is a Coxeter torus in $\SL_n(q)$. 

\medbreak
Indeed, by \eqref{eq:gln-coxeter-torus}, we have
$|\widetilde{T}| = ((q^c)^d - 1) = (q^n - 1)$.
Since $\widetilde{T}$ is cyclic, it is conjugated to $\T_{w}^F$ as claimed after \eqref{eq:gln-coxeter-torus}.  Thus, $|T|=(n)_q$. 

\medbreak
In this subsection we will assume that $\x$ lies in a Coxeter torus $T$ of $\G^F$ arising from some $\y$ as above. 

\begin{prop}\label{prop:sln-not-prime} 
If $\x$ is irreducible, then $\Oc=\oc_{x}^{\PSL_n(q)}$ is of type C.
\end{prop}

\pf Let $n=cd$ with $c$ prime and $d\geq 2$. 
We set:
\begin{align*}
\widetilde{M} &\coloneqq C_{\GL_n(q)}(\y)\simeq \GL_c(q^d);\\
M& \coloneqq C_{\SL_n(q)}(\y) = \widetilde{M} \cap \SL_{n}(q);
\\ 
M_1 &\coloneqq [\widetilde{M},\widetilde{M}]\simeq \SL_c(q^d).
\end{align*}
Thus $M_1 \leq M \leq \widetilde{M}$. 
Lemma  \ref{lem:intersection} gives 
\begin{align*}
\oc_{\x}^{\GL_n(q)} \cap T &= \oc_{\x}^{\SL_n(q)} \cap T = \{\x, \x^q, \dots, \x^{q^j}, \dots, \x^{q^{n-1}} \}\\
\oc_{\x}^{M} \cap T &=\oc_{\x}^{M_1} \cap T=\oc_{\x}^{\widetilde{M}} \cap T=\{\x,\x^{q^d},\ldots,\x^{q^{d(c-1)}}\}.
\end{align*}
Hence $\x^q \in \Ot \cap T$ but $\x^q \notin \oc_{\x}^{M}$. We claim that $\Oc_{\x^q}^{M}\not\subset N_{\widetilde{M}}(T)$. Suppose the contrary.  Then, $\langle \Oc_{\x^q}^{M}\rangle$ would be a non-central,  normal subgroup of $\widetilde{M}\simeq\GL_c(q^d)$. 
Then $\SL_c(q^d)\simeq M_1\leq \langle \Oc_{\x^q}^{M}\rangle \leq N_{\widetilde{M}}(T)$, and so $T \cap M_1$ would be normal in $M_1$, a contradiction. 

\medbreak 
We pick  $\st\in \Oc_{\x^q}^{M}\setminus N_{\widetilde{M}}(T)$ and set 
\begin{align*}s\coloneqq\pi(\st)\in \Oc; &&r\coloneqq\pi(\x)\in \oc; && H \coloneqq\langle r, s,\pi(M_1) \rangle.\end{align*}  

We claim that $r$, $s$ and $H$ satisfy the assumptions of Lemma \ref{lem:equivC}.
First,  $\st\not\in N_{\widetilde{M}}(T)$ implies $s\trid r\neq r$, i.e., \eqref{eq:equivC1} holds. Indeed,  $s\trid r= r$ would give $\st\trid \x\in \Z(\G^F)\x$ that combined with $T=C_{\GL_n(q)}(\x)$ would force $\st\trid T=T$.

\medbreak
In addition,  $\langle \Oc_{\x}^{M_1},\Oc_{\st}^{M_1}\rangle=\langle \Oc_{\x}^{\widetilde{M}},\Oc_{\st}^{\widetilde{M}}\rangle$ is a non-central, normal subgroup of 
$\widetilde{M}\simeq \GL_c(q^d)$,   hence $\langle M_1,\, \x,\, \st\rangle \leq \langle \Oc_{\x}^{M_1},\Oc_{\st}^{M_1}\rangle$ and therefore 
\begin{align*}
H=\langle \pi(M_1),\, r,\, s\rangle \leq \langle \Oc_{r}^{\pi(M_1)},\Oc_{s}^{\pi(M_1)}\rangle\leq \langle \Oc_r^H,\,\Oc_s^H\rangle\leq H.
\end{align*}
That is, \eqref{eq:equivC3} holds and $H \leq \langle\Oc\rangle$.

\medbreak
Observe that $\pi(M)\simeq M/Z(\SL_n(q))\cap M$ onto $\PGL_c(q^d)$,  so the orbits $\Oc_x^{\pi(M)}$ and $\Oc_{x^q}^{\pi(M)}$ project onto non-trivial orbits in 
$\PGL_c(q^d)$,  and therefore $\vert\Oc_r^{H}\vert\geq \vert\Oc_r^{\pi(M)}\vert>2$ and $\vert\Oc_s^{H}\vert\geq\vert\Oc_{s}^{\pi(M)}\vert>2$, i.e., \eqref{eq:equivC4} holds.

\medbreak
We finally analyse $\oc_{r}^{H}\cap\oc_{s}^{H}$.  First of all,  $\pi(M_1)\leq H\leq \pi(M)$ and $\oc_{\x}^{M_1}=\oc_{\x}^{M}$ imply that  
$\oc_{r}^{H}=\Oc_{r}^{\pi(M)}$.  Similarly,  $\oc_{s}^{H}=\Oc_s^{\pi(M)}$.

\medbreak
If  $\Oc_{r}^{\pi(M)}\cap\Oc_s^{\pi(M)}=\emptyset$, then we are done. 
Otherwise,
\begin{align*}
x\in \Oc_s^{\pi(M)}\cap\pi(T)=\Oc_{x^q}^{\pi(M)}\cap\pi(T)=\{x^q, (x^q)^{q^d}, \ldots, (x^q)^{q^{d(c-1)}}\}.
\end{align*}

Therefore there exists  $l\in \I_{0,c-1}$ such that $\x^{q^{dl+1}} \in \G_n(\F_q) \x$.  
Lemma \ref{lema:proj-irr} \ref{item:proj-irr4} gives $l\neq 0$.  
Let $j\in \I_{n-1}$ be minimal satisfying $\x^{q^j}\in \G_n(\F_q) \x$.   
Then $j|n$ by Lemma \ref{lema:proj-irr} \ref{item:proj-irr2} whose argument shows that $j$  divides also $dl+1$. Hence,  $(j,d)=1$.  Since $j>1$ by Lemma \ref{lema:proj-irr} \ref{item:proj-irr4} again,  this can occur only if  $j=c$ and $(c,d)=1$.   
 In this case, $d$ has a prime factor $c'$ different from $c$ and we may 
 repeat the whole construction replacing $c$ by $c'$ and $d$ by $\frac{n}{c'}$.  As $j=c\neq c'$,  we get that $\Oc_{r}^{\pi(M)}\cap\Oc_s^{\pi(M)}=\emptyset$. 
 The hypotheses of Lemma \ref{lem:equivC} were verified, hence $\Oc$ is of type C.
\epf

\subsection{Irreducible elements of \texorpdfstring{$\SL_n(q)$}{}, \texorpdfstring{$n > 3$}{}  prime} \label{subsec:sln-prime}

Here $n > 3$ is prime. Recall that $e \neq x \in \Gb = \PSL_{n}(q)$,
$\x \in \G^F - Z(\G^F)$ is a representative of $x$ which is irreducible and belongs to 
the Coxeter torus $T\coloneqq \T_w^F\cap \SL_n(q)$; we set
$\oc = \oc_{x}^{\Gb}$ and $\Ot = \oc_{\x}^{\G^F}$.
There is an epimorphism of racks $\Ot \twoheadrightarrow \oc$.

\medskip

We will  analyse all possible subgroups of $\GL_n(q)$ intersecting $\Ot$.  We start by a few well-known arithmetic 
results instrumental for our analysis.

\begin{lema}\label{lem:arithmetic} Let $n$ be an odd prime number. 
\begin{enumerate}[leftmargin=*,label=\rm{(\roman*)}]
\item If $(n, q-1)=1$, then $(n,q^n-1)=(n,(n)_q)=1$.
\item If $(n, q-1)=n$, then $(n^2,(n)_q)=(n,(n)_q)=n$. 
\item $(q-1,(n)_q)=(n,q-1)$.
\item  $(n(q-1),\,(n)_q)=(n,\,q-1)$.
\end{enumerate}
\end{lema}
\begin{proof} 
(i) and (ii) are \cite[Lemma 4.1.1]{short}, whilst (iii) follows from the Euclidean algorithm.  We prove (iv).
Combining (i), (ii) and (iii) we obtain \[(n, (n)_q)=(n, q-1)=(q-1,(n)_q),\]  hence
$(n(q-1),\,(n)_q)=1$ if $(n,q-1)=1$,  and $n\leq (n(q-1),\,(n)_q)\leq n^2$ if $(n,q-1)=n$.  In this case,  we discard $ (n(q-1),\,(n)_q)= n^2$ using (ii).
\end{proof}

\medskip 
Recall $\sigma \in N_{\GL_n(q)}(\T_w)$ from Remark \ref{rem:huppert-short-sigma}.

\begin{lema}\label{lem:intersection-N(T)} Let $n$ be a prime. 
\begin{enumerate}[leftmargin=*,label=\rm{(\roman*)}]
\item Let $g\in N_{\GL_n(q)}(\T_w^F)\setminus \T_w^F$. Then $|g|$ divides $n(q-1)$.
\item  $\Ot\cap N_{\G^F}(\T_w)\subset \T_w^F$.
\end{enumerate}
\end{lema}

\begin{proof} (i) By Remark \ref{rem:huppert-short-sigma}, there are $k\in\I_{1,n-1}$ and $t\in \T_w^F$ such that $g = \sigma^kt$.  
Then  
\begin{align*}
(\sigma^k t)^n=\left(\prod_{\tau\in\langle\sigma^k\rangle}\tau\trid t\right)\sigma^{nk}=
\prod_{\tau\in\langle\sigma\rangle}\tau\trid t =\left(\prod_{i=1}^nt^{q^{i}}\right)=t^{(n)_q}
\end{align*} 
by a direct computation. Hence $|g|$ divides $n(q-1)$. 

(ii) Let $g\in\Ot\cap N_{\GL_n(q)}(\T_w^F)=\Ot\cap N_{\G^F}(\T_w^F)$. Recall that $|\x|$ divides $(n)_q$. 
If $g\notin\T_w^F$, then $|g|$ divides $(n(q-1),(n)_q)$ by (i). 
By Lemma \ref{lem:arithmetic} (iv) $\vert g\vert$ divides $(n,q-1)$,  so $g$ is central,  contradicting its irreducibility. 
\end{proof}

\medskip
We recall that a \emph{primitive} prime divisor of $q^n-1$ is a prime number $\ell$ such that $\ell|q^n-1$ and $\ell\not|q^e-1$ for every $e\in\I_{n-1}$

\begin{lema}\label{lem:prime} (Here $n$ is any odd prime). Let $\y$ be an irreducible semisimple element in $\GL_n(q)$. Then,
\begin{enumerate}[leftmargin=*,label=\rm{(\roman*)}]
\item Either there  exists a primitive prime divisor $\ell$ of $q^n-1$ dividing $|\y|$ or else $|\y|$ divides $n(q-1)$,  it does not divide $(q-1)$ and $n| q-1$.
\item If $\y\in\SL_n(q)$, then there always exists a primitive prime divisor $\ell$ of $q^n-1$ dividing $|\y|$.
\end{enumerate}
\end{lema}
\begin{proof}
(i) If for every prime divisor $\ell$ of $|\y|$  there is an $e\in\I_{n-1}$ such that $\ell$ divides  $q^e-1 = (q-1) (e)_q$ then,  
any such $\ell$ divides $(q-1)((n)_q,(e)_q)$ for some $e<n$.  The latter equals $(q-1)(e,n)_q=q-1$ by the Euclidean algorithm, so
 $\ell|q-1$ for any prime divisor $\ell$ of $|\y|$.

Since $\y$ is irreducible,  $\vert\y\vert$ cannot divide $q-1$, so there is a prime divisor $\ell_0$ of $|\y|$ dividing $q-1$ and $(n)_q$.  
By Lemma \ref{lem:arithmetic} (iii) this is possible only if $n|q-1$ and in this case $\ell_0=n$.  
Then, Lemma \ref{lem:arithmetic} (ii) implies that $|\y|$ divides $n(q-1)$. 

\medbreak
(ii) If $\y\in\SL_n(q)$ then $\vert\y\vert$ divides $(n)_q$. Assume, for a contradiction, that no primitive prime divisors of $q^n-1$ divides $|\y|$. 
Then, $|\y|$ would divide $(n(q-1),(n)_q)=(n,q-1)$ by (i) and 
Lemma \ref{lem:arithmetic} (iii). Thus, $\y$ cannot be irreducible.
\end{proof}

\bigbreak

In the terminology of \cite[Definition 1.2]{GPPS}, Lemma \ref{lem:prime}  
says that if $n$ is an odd prime, then all irreducible elements in $\SL_n(q)$  are ${\rm ppd}(n,q;n)$-elements.
The following result is a consequence of  \cite{GPPS}.

\begin{lema}\label{lem:subgroupsGPPS}Let  $\ell$ be a primitive prime divisor of $q^n-1$ dividing $|\x|$ and 
let $H\leq \GL_n(q)$ be such that $\x\in H$. Then $H$ occurs in the following list.
\begin{enumerate}[leftmargin=*,label=\rm{(\alph*)}]
\item\label{item:511a} $\SL_n(q_0)\leq H\leq N_{\GL_n(q)}(\SL_n(q_0))$ where $q_0=p^{m_0}$ 
with $ m= m_0 d$,  $d\in \N$ and $(d,n)=1$.

\smallbreak
\item\label{item:511b} $\SU_n(q^{1/2}_0)\leq H\leq N_{\GL_n(q)}(\SU_n(q_0))$ where $q_0=p^{m_0}$  
a square with $ m= m_0 d$,  $d\in \N$ and $(d,n)=1$.

\smallbreak
\item\label{item:511c} $H\leq N_{\GL_n(q)}({\T}^F_w)=N_{\GL_n(q)}(T)$, and $\ell$ divides $|H|$.

\smallbreak
\item\label{item:511d} $H/(H\cap Z(\GL_n(q)) \simeq M_{11}$, $n=5$,  $\ell=11$, and  $q^5  \equiv 1\mod 11$.

\smallbreak
\item\label{item:511e} $H/(H\cap Z(\GL_n(q)) \simeq M_{23}$, or $M_{24}$, $n=11$,  $\ell=23$,  
$q^{11}  \equiv 1\mod 23$.

\smallbreak
\item\label{item:511f} $\PSL_2(\ell)\leq H/(H\cap Z(\GL_n(q)))\leq \PGL_2(\ell)$, for $\ell\geq 7$, $n=\frac{1}{2}(\ell-1)$ and $q^{n}  \equiv 1\mod \ell$. 
\end{enumerate}
\end{lema}
\noindent \emph{Proof.\/}
The main result in  \cite{GPPS} states that the subgroups of $\GL_d(q)$ 
containing a ${\rm ppd}(d,q;e)$-element, for some $\frac{1}{2}d<e\leq d$ 
are precisely those occurring in the Examples 2.1, \dots, 2.9 listed therein. 
We extract the cases satisfying $d=e=n$ an odd  prime $>3$.  

\begin{itemize}[leftmargin=*]\renewcommand{\labelitemi}{$\circ$}
\item Example 2.1 (b) and (d) 
and Example 2.5 are discarded because they occur for  either $d$ or $e$ even. 

\smallbreak\item
Examples 2.1 (a) and (c) are (a) and (b) in our list.  

\smallbreak\item 
Example 2.2 does not occur because 
it requires an $H$-stable subspace of the natural representation of $\GL_n(q)$ and $\x\in H$ is irreducible. 

\smallbreak\item
Examples 2.3 and Examples 2.4 (a)  are discarded as they require $e\neq d$.  

\smallbreak\item
Example 2.4 (b) is the case (c) in our list.  

\smallbreak\item 
Example 2.6 (a) is discarded because 
it requires the prime $\ell=n+1$, which is impossible because $n>2$. 

\smallbreak\item
Examples 2.6 (b) and (c) are collected in \cite[Tables 2,3,4]{GPPS}. In Table 2,  $d$ is even. In Tables 3 and 4  the number $e$ is  never odd $>3$. 

\smallbreak\item
Examples  2.7 are listed in \cite[Table 5]{GPPS}.
The column with $\ell=e+1$, is immediately discarded, just as all rows 
for which $d$ is not a prime number. We are left with the three possible choices 
for $H'\simeq H/H\cap Z(\GL_n(q))$ listed in (d) and (e).

\smallbreak\item Examples 2.8 are listed in  \cite[Table 6]{GPPS}  and are discarded 
because $e\neq d$. 

\smallbreak\item
Examples 2.9 are listed in \cite[Tables 7,8]{GPPS}
and if $d$ is a prime, then $e$ is even. In Table 8, we discard all cases 
for which $\ell=e+1$ and we are left with the case (f) in our list.  \qed
\end{itemize}

\begin{lema}\label{lem:special-intersections}
Let $\G=\SL_n(\kk)$,  $m=m_0d$ with $(d,n)=1$ and $q_0=p^{m_0}$.
\begin{enumerate}[leftmargin=*,label=\rm{(\roman*)}]
\item\label{item:inclusion-SL} If $(n,q_0-1)=n$, then $Z(\GL_n(q))\GL_n(q_0)\cap \SL_n(q)=\SL_n(q_0)$.

\medbreak
\item\label{item:inclusion-SU}If $m_0$ is even and $(n,q_0^{1/2}+1)=n$, then  \[Z(\GL_n(q))\GU_n(q^{1/2}_0)\cap \SL_n(q)=\SU_n(q^{1/2}_0).\]
\end{enumerate}
\end{lema}
\begin{proof}
\ref{item:inclusion-SL} We  prove $\subset$.  Let  $\zeta\id_n\in  Z(\GL_n(q))$
and $g\in \GL_n(q_0)$ be such that $\zeta g\in \SL_n(q)$. 
Now, $|\zeta|$ divides $n(q_0-1)$ because $\zeta^{-n}=\det(g)\in\F_{q_0}^\times$.  It also divides $q-1$ because $\zeta\in\F_q^\times$.
Hence it divides  \begin{align*}(n(q_0-1),\,q-1)=(q_0-1)\left(n,(d)_{q_0}\right).\end{align*} 

Hence, $(d)_{q_0}\equiv d\mod n$;  since $(d,n) = 1$, then $|\zeta|$ divides $q_0-1$, i.e., $\zeta\in \F^\times_{q_0}$ and $\zeta g\in\SL_n(q_0)$.  

\medbreak
\ref{item:inclusion-SU}  We  prove $\subset$.  Let $\zeta\id_n\in Z(\GL_n(q))$ and $g\in \GU_n(q^ {1/2}_0)$ 
be such that $(\zeta\id_n)g\in\SL_n(q)$.  Now,  as $g=\Fr_{q_0^{1/2}}\phi(g)$, we have $(\det g)^{q_0^{1/2}+1}=1$.  
Hence $|\zeta|$ divides $n(q_0^{1/2}+1)$ and also $q-1$ because $\zeta\in\F_q^\times$, and so it divides  \begin{align*}(n(q_0^{1/2}+1),\,q-1)=(q_0^{1/2}+1)\left(n,(q_0^{1/2}-1)(d)_{q^{1/2}_0}\right). \end{align*} 
However, $n$ is an odd prime dividing $q_0^{1/2}+1$ so it does not divide $q_0^{1/2}-1$; also, $(n, (d)_{q_0^{1/2}})=1$ 
by the argument in \ref{item:inclusion-SL} applied to $q_0^{1/2}$. Thus,  $|\zeta|$ divides $q^{1/2}_0+1$, that is, 
$\zeta\id_n\in Z(\GU_n(q_0^{1/2}))$, so $\zeta g\in\SL_n(q)\cap\GU_n(q_0^{1/2})=\SU_n(q_0^{1/2})$.
\end{proof}

\medskip

In this Subsection $\widetilde{\pi}\colon \GL_n(q)\to \PGL_n(q)$ 
is the natural projection, whose restriction to $\SL_n(q)$ is $\pi$. 

\begin{prop}\label{prop:irreducible-kthulhu}
Let $\x$ be an irreducible element in the Coxeter torus $T=\T_w^F$. Then  $\Oc$ is kthulhu.
\end{prop}
\noindent \emph{Proof. \/}
We consider all possible intersections $\Oc\cap {\bf M}$ for every  ${\bf M}\leq \Gb$ containing $x$. All such groups have the form ${\bf M}={\pi}(H\cap\SL_n(q))$ for some $H\leq\GL_n(q)$ containing $\x$. We will show that either $\Oc\cap {\bf M}=\Oc_x^{\bf M}$ or else $\Oc\cap {\bf M}$ is an abelian subrack.  This implies that $\Oc$ is kthulhu.

\medbreak

For our analysis, we will make use of the following auxiliary facts:

\begin{claimsln} 
$C_{\PSL_n}(x)\cap \Oc=\{x,x^q,\,\ldots,\,x^{q^{n-1}}\}$.
\end{claimsln}
Indeed,  Lemma  \ref{lem:intersection} gives 
\[C_{\SL_n(q)}(\x)\cap \Ot=\{\x,\x^q,\,\ldots,\,\x^{q^{n-1}}\}.\] We describe $C_{\PSL_n(q)}(x)$. If $\z\in \SL_n(q)$ satisfies $\z\x\z^{-1}\in \G_m(\F_q)\x\cap\Ot$, then by primality of $n$ and Lemma \ref{lema:proj-irr} \ref{item:proj-irr2} and \ref{item:proj-irr4} we conclude that $\z\x\z^{-1}=\x$, and so $ C_{\PSL_n(q)}(x)=\pi(C_{\SL_n(q)}(\x))=\pi(T)$, whence the claim. 

\begin{claimsln}\label{claim:sln2}  If $\widetilde{\pi}(H)$ is simple, then $\widetilde{\pi}(H)\leq \PSL_{n}(q)$.  In particular,  we may assume $H\leq\SL_n(q)$. \end{claimsln}

Indeed, if $\widetilde{\pi}(H)$ is simple, then \[\widetilde{\pi}(H)=[\widetilde{\pi}(H),\widetilde{\pi}(H)]=\widetilde{\pi}([H,H])\leq \widetilde{\pi}([\GL_n(q),\GL_n(q)])=\pi(\SL_n(q)).\]

\bigbreak
We set from now on $H_1\coloneqq H\cap\SL_n(q)$ and inspect all possible ${\bf M}={\pi}(H_1)$ where  $H$ runs through the list  of subgroups from Lemma \ref{lem:subgroupsGPPS} containing $\x$, with $\ell$ a primitive prime divisor of 
$q^n - 1$ dividing $|\x|$.  The numbering of items is as in Lemma \ref{lem:subgroupsGPPS}.  

\medbreak

\noindent \emph{Case }\ref{item:511a}.   Here $q=p^m$, $q_0=p^{m_0}$ where $ m= m_0 d$,  $d\in \N$ 
 and $(n, d)=1$. Proposition \ref{prop:normaliserSLq0} gives $N_{\GL_n(q)}(\SL_n(q_0))=Z(\GL_n(q))\GL_n(q_0)$ so 
\begin{align}\label{eq:inclusions}\SL_n(q_0)\leq H_1\leq Z(\GL_n(q))\GL_n(q_0)\cap \SL_n(q).\end{align} 
We will first show that 
\begin{align}\label{eq:SLnq0}
\Ot\cap Z(\GL_n(q))\GL_n(q_0)\cap \SL_n(q)=\Oc_{\x}^{\SL_n(q_0)}.
\end{align}

If $(n,q_0-1)=n$,  then the inclusions in \eqref{eq:inclusions} are all equalities by Lemma \ref{lem:special-intersections} (i).   In this case 
\begin{align*}
\Ot\cap Z(\GL_n(q))\GL_n(q_0)\cap \SL_n(q) &= \Ot \cap \SL_n(q_0) 
\\ &= \Oc_{\x}^{\SL_n(\kk)}\cap \SL_n(q_0)= \Oc_{\x}^{\SL_n(q_0)}
\end{align*}
where the last two equalities follow from  \cite[Theorem 21.11]{MT} and \cite[\S 2.11]{hu-cc}.
\medskip 
Assume now that $(n,q_0-1)=1$.  Since $\x\in H_1\leq Z(\GL_n(q))\GL_n(q_0)$,  there are $z=\zeta\id_n\in Z(\GL_n(q))$ and $\y\in \GL_n(q_0)$ such that
\begin{align*}
\x=z\y.
\end{align*}
Consider $\x_1\in \Ot\cap Z(\GL_n(q))\GL_n(q_0)$.   Let $z_1=\zeta_1\id_n\in Z(\GL_n(q))$ and 
$\y_1\in \GL_n(q_0)$ be such that $\x_1=z_1\y_1$.  By construction, $|\zeta|$ and $|\zeta_1|$ divide $n(q_0-1)$ because $\x,\,\x_1\in\SL_n(q)$.  
Since $\x$ is irreducible, $\y$ and $\y_1$ are again irreducible in $\GL_n(q)$, whence in $\GL_n(q_0)$, 
because they are regular and lie in a Coxeter torus of $\GL_n(q)$.
Let $\{\eta^{q^j}: j\in\I_{0,n-1}\}\subset \F_{q_0^n}$ and  $\{\eta_1^{q^j}: j\in\I_{0,n-1}\}\subset\F_{q_0^n}$  be 
the sets of eigenvalues of $\y$ and $\y_1$, respectively,  so
$\{\zeta\eta^{q^j}:  j\in\I_{0,n-1}\}$ and $\{\zeta_1\eta_1^{q^j}: j\in\I_{0,n-1}\}$ are the 
sets of eigenvalues of $\x$ and $\x_1$, respectively. Then
\[\{\zeta\eta^{q^j}: j\in\I_{0,n-1}\}=\{\zeta_1\eta_1^{q^j}: j\in\I_{0,n-1}\}\]
and so $\zeta\eta=\zeta_1\eta_1^{q^{j_0}}$ for some $j_0$. Therefore
$|\zeta_1\zeta^{-1}|=|\eta\eta_1^{-q^{j_0}}|$ divides $\left(n(q_0-1), q_0^n-1\right)=(q_0-1)(n,(n)_{q_0})=q_0-1$,
where the last equality follows from Lemma \ref{lem:arithmetic} (i). 
In other words, $\zeta_1\in \zeta\F_{q_0}^\times$, and $z^{-1}\x_1$ is a regular semisimple matrix in $\GL_n(q_0)$
with the same eigenvalues as $\y$, and it is therefore $\SL_n(q_0)$-conjugate to $\y$.  Hence,  
\begin{align*}
\Ot\cap Z(\GL_n(q))\GL_n(q_0)\subset z\Oc_{\y}^{\SL_n(q_0)}=\Oc_{\x}^{\SL_n(q_0)}.
\end{align*}

\medskip

Let  now $x'=\pi(\x')\in\Oc\cap{\bf M}$. Then, $z'\x'\in \Ot$ for some  $z'\in Z(\SL_n(q))$ and $\x'\in Z(\SL_n(q))H_1$, that is, 
\begin{align*}
z'\x'\in\Ot\cap Z(\SL_n(q))H_1\subset \Ot\cap Z(\GL_n(q))\GL_n(q_0)\cap \SL_n(q)=\Oc_{\x}^{H_1}.
\end{align*}
where the equality follows from \eqref{eq:inclusions}  and \eqref{eq:SLnq0}.
Thus, $\x'\in Z(\SL_n(q))\Oc_{\x}^{H_1}$ and $x'\in\Oc_{x}^{\pi(H_1)}=\Oc_{x}^{\bf M}$, showing that $\Oc\cap{\bf M}=\Oc_x^{\bf M}$.

\bigbreak

\noindent  \emph{Case }\ref{item:511b}. Here $q=p^m$, $q_0=p^{m_0}$ where $m_0|m$, $m_0$ is even and $(n, d)=1$. 
We use the same strategy as in case \ref{item:511a}. 

Proposition \ref{prop:normaliserSLq0} gives $N_{\GL_n(q)}(\SU_n(q^{1/2}_0))=Z(\GL_n(q))\GU_n(q^{1/2}_0)$ so 
\begin{align}\label{eq:inclusions2}\SU_n(q^{1/2}_0)\leq H_1\leq Z(\GL_n(q))\GU_n(q^{1/2}_0)\cap \SL_n(q).\end{align}

We will first show that 
\begin{align}\label{eq:SUnq0}
\Ot\cap Z(\GL_n(q))\GU_n(q^{1/2}_0)\cap \SL_n(q)=\Oc_{\x}^{\SU_n(q^{1/2}_0)}.
\end{align}

If $(n,q_0^{1/2}+1)=n$,  then the inclusions in \eqref{eq:inclusions2} are all equalities by Lemma \ref{lem:special-intersections} (ii).   In this case 
\begin{align*}
\Ot\cap Z(\GL_n(q))\GU_n(q^{1/2}_0)\cap \SL_n(q) &= \Ot \cap \SU_n(q^{1/2}_0) 
\\ &= \Oc_{\x}^{\SL_n(\kk)}\cap \SU_n(q^{1/2}_0)= \Oc_{\x}^{\SU_n(q^{1/2}_0)}
\end{align*}
where the last two equalities follow from  \cite[Theorem 21.11]{MT} and \cite[\S 2.11]{hu-cc}.

\medskip 

Assume now that $(n,q_0^{1/2}+1)=1$.  Since $\x\in H_1\leq Z(\GL_n(q))\GU_n(q^{1/2}_0)$,  
there are $z=\zeta\id_n\in Z(\GL_n(q))$ and $\y\in \GU_n(q_0)\leq\GL_n(q_0)$ such that
\begin{align*}
\x=z\y.
\end{align*}
Consider $\x_1\in \Ot\cap Z(\GL_n(q))\GU_n(q_0^{1/2})$.  Let $z_1=\zeta_1\id_n\in Z(\GL_n(q))$ and $\y_1\in \GU_n(q_0)\leq\GL_n(q_0)$ be such that $\x_1=z_1\y_1$.  
By construction, $|\zeta|$ and $|\zeta_1|$ divide $n(q_0^{1/2}+1)$ because 
$\x,\,\x_1\in\SL_n(q)$ and $\det(g)^{q_0^{1/2}+1}=1$ for any $g\in \GU_n(q_0^{1/2})$. 

Since $\x$ is irreducible, $\y$ and $\y_1$ are again irreducible in $\GL_n(q)$, whence in $\GL_n(q_0)$.  
We show that $|\y|$ cannot divide $n(q_0-1)$. 
Indeed,  if this were the case, then we would have $\y^{q_0-1}=\xi\id_n$ for some $\xi\in\G_n(\F_q^\times)$, 
with $\xi\neq 1$.  Since $\y^{q_0}\in\Oc_{\y}^{\GL_n(q_0)}$,
the characteristic polynomial $p_{\y}$ would be $X^n-\det(\y)=X^n-\xi^{-n}$,  which is not irreducible.  
Hence, by Lemma \ref{lem:prime} (i)  there is a primitive prime divisor $\ell$ of $|\y|$ dividing $q_0^n-1$.  

\medbreak Let $F_0\colon \GL_n(\kk)\to\GL_n(\kk)$ be given by 
$F_0(A)\coloneqq \Fr_{q_0^{1/2}}\phi(A)$,  for $A\in \GL_n(\kk)$, cf. Subsection \ref{subsec:normalizers}.
By \cite[Proposition 26.6]{MT} there exists an $F_0$-stable  torus $\T'$
in $\GL_n(\kk)$ containing $\y$.
Let $\mathcal T = \T' \cap \GU_n(q_0^{1/2})$.  
By \cite[Proposition 25.3 (c)]{MT} and an analysis of $\phi$-classes in the symmetric group, 
 there is a partition $\lambda$ of $n$ such that 
\[|{\mathcal T}|=\prod_{\lambda_i \textrm{ even }}(q_0^{\lambda_i/2}-1)\prod_{\lambda_i \textrm{ odd }}(q_0^{\lambda_i/2}+1). \] 
The latter divides \(\prod_{\lambda_i \textrm{ even }}(q_0^{\lambda_i/2}-1)\prod_{\lambda_i \textrm{ odd }}(q_0^{\lambda_i}-1)\)
and is divisible by the primitive prime divisor $\ell$ of $q_0^n-1$. Hence, $\lambda=(n)$ and $|{\mathcal T}|=(q_0^{n/2}+1)$.

Now we proceed as in case \ref{item:511a}: considering the set of eigenvalues for $\x$ and
$\x_1$ and of $\y$ and $\y_1$, we deduce that
$|\zeta_1\zeta^{-1}|$ divides 
\[\left(n(q^{1/2}_0+1), q_0^{n/2}+1\right)=(q^{1/2}_0+1)(n,(n)_{-q^{1/2}_0})=(q^{1/2}_0+1) \]
where $(n,(n)_{-q_0^{1/2}})=1$ because $(n)_{-q_0^{1/2}}$
divides $q_{0}^{n/2}+1$ and $q_0^{1/2}$ is not a root of $X^n+1=(X+1)^n$ in $\F_n$ by our assumption on $q_0$ and $n$.  

Hence, $\z_1\in \z Z(\GU_n(q_0^{1/2}))$, and $z^{-1}\x_1$ is a regular semisimple matrix in $\GU_n(q^{1/2}_0)$ 
with the same eigenvalues as $\y$, and it is therefore $\SU_n(q_0)$-conjugate to $\y$ by  \cite[\S 2.11, \S 8.5]{hu-cc}. Hence,  
\begin{align*}
\Ot\cap Z(\GL_n(q))\GU_n(q^{1/2}_0)\subset z\Oc_{\y}^{\SU_n(q^{1/2}_0)}=\Oc_{\x}^{\SU_n(q^{1/2}_0)}.
\end{align*}

Let  now $x'=\pi(\x')\in\Oc\cap{\bf M}$. Then, $z'\x'\in \Ot$ for some  $z'\in Z(\SL_n(q))$ and $\x'\in Z(\SL_n(q))H_1$, that is, 
\begin{align*}
z'\x'\in\Ot\cap Z(\SL_n(q))H_1\subset \Ot\cap Z(\GL_n(q))\GU_n(q^{1/2}_0)\cap \SL_n(q)=\Oc_{\x}^{H_1}.
\end{align*}
where the equality follows from \eqref{eq:inclusions2}  and \eqref{eq:SUnq0}.
Thus, $\x'\in Z(\SL_n(q))\Oc_{\x}^{H_1}$ and $x'\in\Oc_{x}^{\pi(H_1)}=\Oc_{x}^{\bf M}$, showing that $\Oc\cap{\bf M}=\Oc_x^{\bf M}$.

\bigbreak

\noindent \emph{Case }\ref{item:511c} In this case, ${\bf M}\leq \pi(N_{\SL_n(q)}(T))=N_{\Gb}(\pi(T))$, where the second equality follows because $Z(\SL_n(q))\leq T$. If $y=\pi(\y)\in \Oc\cap{\bf M}$ then there is $\z\in Z(\SL_n(q))$ such that 
$\y\in \Oc_{\z\x}^{\SL_n(q)}\cap N_{\SL_n(q)}(T)$, and $\z\x$ is again irreducible. By Lemma \ref{lem:intersection-N(T)} we see that $\y\in T$, so $\Oc\cap{\bf M}\subset \pi(T)$ is abelian.

\bigbreak

\noindent \emph{Case }\ref{item:511d} In this case $n=5$ and $\ell=11$ and $\widetilde{\pi}(H)=M_{11}=\pi(H_1)$. 
We show that $\Oc\cap M_{11}=\Oc_x^{M_{11}}$. The only elements whose order is divisible by  $\ell$ in $M_{11}$ are of order $11$, so $|x|=11$. There are two classes of such elements in $M_{11}$, say $\Oc_x^{M_{11}}$ and  $\Oc_y^{M_{11}}$. If   $\Oc\cap M_{11}=\Oc_x^{M_{11}}\cup \Oc_y^{M_{11}}$, then $\langle x\rangle-1\subset \Oc\cap M_{11}\cap C_{\PSL_n(q)}(x)=\{x,x^q,x^{q^2},x^{q^3},x^{q^4}\}$, a contradiction.

\medbreak

\noindent \emph{Case }\ref{item:511e} In this case $n=11$ and $\ell=23$. The only elements of order divisible by $\ell$ in ${\bf M}=M_{23}$ or $M_{24}$  have order $23$ and there are $2$ conjugacy classes of such elements.
We proceed as in  case \ref{item:511c}.  

\medbreak

\noindent \emph{Case }\ref{item:511f} In this case $\ell|q^n-1$ and ${\bf M}=\pi(H_1)\simeq H_1/H_1\cap Z(\SL_n(q))$, and
\[\PSL_2(\ell)\leq H_1/H_1\cap Z(\SL_n(q))\leq\PGL_2(\ell).\] 
As $[\PGL_2(\ell):\PSL_2(\ell)]\leq2$, we have ${\bf M}\simeq \PGL_2(\ell)$ or  ${\bf M}\simeq \PSL_2(\ell)$. In both cases, $|x|=\ell$ and we claim that $\Oc\cap{\bf M}=\Oc_x^{\bf M}$.
In $\PGL_2(\ell)$ all non-trivial unipotent elements are conjugate and the claim follows.  Let $y\in \Oc\cap\PSL_2(\ell)$. By replacing $y$ with a representative lying in the same Borel subgroup of $\PSL_2(\ell)$ as $x$, we can ensure that \[y\in C_{\PSL_2(\ell)}(x)\cap \Oc\subset C_{\PSL_n(q)}(x)\cap \Oc=\{x,x^q,\,\ldots,\,x^{q^{n-1}}\}.\] Without loss of generality we may assume that 
$x$ is the class of $\left(\begin{smallmatrix}
1&\xi\\
0&1\end{smallmatrix}\right)$, for some $\xi\in \F_\ell^\times$ so $y$ is the class of $\left(\begin{smallmatrix}
1&\xi\\
0&1\end{smallmatrix}\right)^{q^j}=\left(\begin{smallmatrix}
1&q^j\xi\\
0&1\end{smallmatrix}\right)$ for some $j\in \I_{n-1}$.  By assumption $q\equiv q^{n+1}\mod\ell$ hence $q$ is a square modulo $\ell$. Therefore $x$ and $y$ are conjugate in $\PSL_2(\ell)$, whence the claim. 
\qed

\begin{obs}
Consider either $g\in M_{11}$, $\vert g \vert =11$, or $g\in M_{23}$ or $M_{24}$, $\vert g \vert = 23$.
Then the classes $\Oc_{g}^{M_{11}}$, $\Oc_{g}^{M_{23}}$ or $\Oc_{g}^{M_{24}}$
are contained either in $\Oc_{g}^{\PSL_5(q)}$ for some $q$, or in $\Oc_{g}^{\PSL_{11}(q')}$ for some $q'$,
respectively, according to \cite{GPPS} and Claim \ref{claim:sln2}, see  Cases \ref{item:511d} and \ref{item:511e}.
By Proposition \ref{prop:irreducible-kthulhu}, since $g$ is irreducible in all cases, $\Oc_{g}^{\Gb}$ is kthulhu,
where $\Gb$ is either  $\PSL_5(q)$ or  $\PSL_{11}(q')$. Hence so are 
$\Oc_{g}^{M_{11}}$, $\Oc_{g}^{M_{23}}$ and $\Oc_{g}^{M_{24}}$, as
 was previously proved in \cite[Teorema 3.26]{Beltran}.
\end{obs}

\section{Semisimple conjugacy classes represented in \texorpdfstring{$K$}{}}\label{sec:K}

In this section we deal  with  semisimple conjugacy classes intersecting the subgroup $K$ 
which is the image of the map $j\colon\GL_n(q)\to \G^F$ introduced in \S \ref{subsec:K}. 
We give parallel proofs for two classes of simple groups:

\medbreak
\begin{itemize}[leftmargin=*]\renewcommand{\labelitemi}{$\circ$}
\item $\G = \Sp_{2n}(\kk)$ with $n\geq 2$; here $\Gb\coloneqq \G^F/Z(\G^F)$ and 
$\pi\colon \G^F\to\Gb$ denotes the standard projection.

\medbreak
\item $\G = \SO_{n'}(\kk)$ with $n'=2n$ and $n\geq 4$,
or $n'=2n+1$ and $n \geq 3$; here
$\Gb\coloneqq [\G^F, \G^F]/Z(\G^F)$ and  $\pi\colon [\G^F, \G^F]\to\Gb$ is the standard projection. 
\end{itemize} 

\medbreak
In the symplectic case, $\G^F = [\G^F, \G^F]$ so for brevity of the exposition we write $[\G^F, \G^F]$ in both cases.
We also consider such groups with smaller $n$ sometimes for the sake of recursive arguments.

\medbreak
We shall consider a semisimple class $\Oc$   in $\Gb$, a  class $\Ot$   in $[\G^F, \G^F]$ such that $\pi(\Ot) = \Oc$
and assume that it exists $A\in\GL_n(q)$ such that $j(A) \in \Ot$.

\medbreak
Here are the main results of this Section:

\begin{theorem}\label{thm:section-meetsK-PSp}
Let $\G = \Sp_{2n}(\kk)$, $n\geq 2$, and let $A\in \GL_n(q)-Z(\GL_n(q))$ 
be a semisimple element, which is not an involution if $n=2$ and $q\leq 7$.
Then $\Oc=\Oc_{\pi (j(A))}^{\Gb}$ collapses. 
\end{theorem}

\begin{theorem}\label{thm:section-meetsK-PSO}
Let $\G=\SO_{2n}(\kk)$  or $\SO_{2n+1}(\kk)$ with $n\geq 3$  in both cases
and let $A\in \GL_n(q)-Z(\GL_n(q))$ be a semisimple element.

Assume in addition that $j(A)$ does not correspond to situation \eqref{eq:special-situation} if $q\in\{3,5,7\}$. 
Then $\Oc=\Oc_{\pi (j(A))}^{\Gb}$ collapses.
\end{theorem}

These theorems are proved in Subsection \ref{pfthm:section-meetsK} 
after we deal in Subsection \ref{subsec:meetsK-irreducible} with the case when $A$ is irreducible.

\medbreak
In the orthogonal case, we consider the orbit $\Oc_{j(A)}^{[\G^F, \G^F]}$ for later applications
even if $j(A)$ does not necessarily belong to $[\G^F,\G^F]$, as in Remark \ref{obs:normal-subrack}.
See Lemmata \ref{lema:spso-gl-irr} and \ref{lema:sp-gl-irr}.

\medbreak
We start by some general considerations.

\begin{lema}\label{lema:properties-j} Let $A\in\GL_n(q)$ be a semisimple element.
\begin{enumerate}[leftmargin=*,label=\rm{(\roman*)}]
\item\label{item:K1} 
$\Oc_{j(A)}^{j(\SL_n(q))}=\Oc_{j(A)}^{[K,K]}=\Oc_{j(A)}^K=\Oc_{j(A)}^{j(\GL_n(q))}$.

\medbreak
\item\label{item:K2} If $A$ is irreducible, then either $\Oc_{A}^{\GL_n(q)}=\Oc_{A^{-1}}^{\GL_n(q)}$ or else $j(A)$ is regular in $\GL_{n'}(q)$.

\medbreak
\item\label{item:K3} If $A$ is irreducible, then  either  $\Oc_{j(A)}^{[\G^F,\G^F]}=\Oc_{j(A^{-1})}^{[\G^F,\G^F]}$, or else $j(A)$ is regular. 
\end{enumerate}
\end{lema}

\pf 

\ref{item:K1}  is a consequence of the inclusions \[j(\SL_n(q))\simeq [K,K]\leq K\simeq \GL_n(q).\] 
\ref{item:K2}: If $\zeta^{q^h}$, $h\in\I_{0,n-1}$, are the (distinct)  eigenvalues of $A$ in $\kk$, then $\zeta^{\pm q^h}$ for $h\in\I_{0,n-1}$ 
(together with $1$ when $n'=2n+1$) are the eigenvalues of $j(A)$. 
Assume that $j(A)$ is not regular in $\GL_{n'}(q)$; hence $A$ and $A^{-1}$ have a common eigenvalue. 
Then the sets of eigenvalues of $A$ and $A^{-1}$ coincide by irreducibility, 
that is $A$ and $A^{-1}$ are conjugate in $\GL_{n}(\kk)$. 
Since centralisers in $\GL_{n}(\kk)$ are connected,  \cite[p. 19]{hu-cc}, $\Oc_{A}^{\GL_n(q)}=\Oc_{A^{-1}}^{\GL_n(q)}$.
\ref{item:K3} follows from \ref{item:K1} and \ref{item:K2} and the inclusion $[K,K]\leq K\cap[\G^F,\G^F]$.
\epf

\subsection{\texorpdfstring{$A\in \GL_n(q)$}{} is irreducible}\label{subsec:meetsK-irreducible}
We first analyze this case.

\begin{lema}\label{lema:spso-gl-irr} (Here $n\geq 3$ when $\G = \SO_{2n}(\kk)$ or $\G = \SO_{2n+1}(\kk)$).
Let $A\in \GL_n(q)$ be an irreducible element   such that 
\begin{itemize}[leftmargin=*]\renewcommand{\labelitemi}{$\circ$}
\item $j(A)$  is not regular in $\GL_{n'}(q)$,

\item $p_A (X)\neq X^2+1$ when  $\G=\Sp_{4}(\kk)$ and $q\equiv 3\mod 4$.
\end{itemize}
Then
\begin{enumerate}[leftmargin=*,label=\rm{(\roman*)}] 
\item $\Os \coloneqq \Oc_{j(A)}^{[\G^F,\G^F]}$ collapses. 

\item   If $j(A) \in \Ot \subseteq [\G^F,\G^F]$, then $\Oc$ collapses. 
\end{enumerate}
\end{lema}
\noindent\emph{Proof. \/} Observe that {\rm (ii)} follows directly from {\rm (i)}, that we prove. 
The initial discussion is valid  for both orthogonal and symplectic groups. Recall $\phi$ from \eqref{eq:def-phi}.
The irreduciblity assumption in $A$ ensures that the eigenvalues of  $A$ are all distinct, i.e., $A$ is regular semisimple,  
so we may assume that $A$ is the companion matrix of its characteristic  (and minimal) polynomial 
$p_A=X^n+a_{n-1}X^{n-1}+\cdots a_0$.
That is, $A$, $\phi(A)$, $^tA$, $^tA^{-1}$, $A^{-1}$, and $\phi(A^{-1})$ have the following shape:

\begin{align*}
&A=\left(\begin{smallmatrix}
0&0&\cdots&0&-a_0\\
1&0&\cdots&0&-a_1\\
0&1&\cdots&0&-a_2\\
\cdots&\cdots&\cdots&\cdots&\cdots\\
0&0&\cdots&1&-a_{n-1}
\end{smallmatrix}\right),& &\phi(A)=\left(\begin{smallmatrix}
0&1&0&\cdots&0\\
0&0&1&\cdots&0\\
\cdots&\cdots&\cdots&\cdots&\cdots\\
0&\cdots&\cdots&0&1\\
-1/a_0&-a_{n-1}/a_0&\cdots&-a_2/a_0&-a_{1}/a_0
\end{smallmatrix}\right),
\end{align*}
\begin{align*}
&^t\!A=\left(\begin{smallmatrix}
0&1&0&\cdots&0\\
0&0&1&\cdots&0\\
\cdots&\cdots&\cdots&\cdots&\cdots\\
0&0&0&\cdots&1\\
-a_0&-a_1&-a_2&\cdots&-a_{n-1}
\end{smallmatrix}\right),&&
^t\!A^{-1}=\left(\begin{smallmatrix}
-a_1/a_0&-a_2/a_0&\cdots &-a_{n-1}/a_0&-1/a_0\\
1&0&\cdots&0&0\\
0&1&\cdots&0&0\\
\cdots&\cdots&\cdots&\cdots&\cdots\\
0&0&\cdots&1&0
\end{smallmatrix}\right),
\end{align*}
\begin{align*}
&A^{-1}=\left(\begin{smallmatrix}
-a_1/a_0&1&0&\cdots&0\\
-a_2/a_0&0&1&\cdots&0\\
\cdots&\cdots&\cdots&\cdots&\cdots\\
-a_{n-1}/a_0&0&0&\cdots&1\\
-1/a_0&0&0&\cdots&0
\end{smallmatrix}\right),&&
\phi(A^{-1})=\left(\begin{smallmatrix}
-a_{n-1}&-a_{n-2}&\cdots&\cdots&-a_0\\
1&0&0&\cdots&0\\
0&1&0&\cdots&0\\
\cdots&\cdots&\cdots&\cdots&\cdots\\
0&0&0&1&0
\end{smallmatrix}\right).\\
\end{align*}
Also $A\neq A^{-1}$, otherwise $A$ would have eigenvalues $\pm1$, contradicting irreducibility.
We consider the  disjoint subracks:
\begin{align*}
R_1&\coloneqq\left\{\left(\begin{smallmatrix}
A&Y\\
0&\phi(A)\end{smallmatrix}\right)\in\Os\right\},&R_2&\coloneqq\left\{\left(\begin{smallmatrix}
A^{-1}&Y\\
0&\phi(A^{-1})\end{smallmatrix}\right)\in\Os\right\},&\text{if } n'&=2n;
\\
R_1&\coloneqq\left\{\left(\begin{smallmatrix}
A&0&Y\\
0&1&0\\
0&0&\phi(A)\end{smallmatrix}\right)\in\Os\right\},&
R_2 &\coloneqq\left\{\left(\begin{smallmatrix}
A^{-1}&0&Y\\
0&1&0\\
0&0&\phi(A^{-1})\end{smallmatrix}\right)\in\Os\right\}, &\text{if } n'&=2n+1.
\end{align*}
Now $R_1\neq\emptyset$ by construction, and $R_2\neq\emptyset$ 
by Lemma \ref{lema:properties-j} \ref{item:K3}.  It is easy to see that $R_i\trid R_j=R_j$, $1\leq i,j \leq 2$.
We  continue with each group separately. 

\begin{step}
$\G=\Sp_{2n}(\kk)$. Let
\begin{align*}
r_1&\coloneqq\left(\begin{smallmatrix}
\id_n&\Jf_n\\
0&\id_n\end{smallmatrix}\right)\trid j(A)=\left(\begin{smallmatrix}
A&-A\Jf_{n}+^t\!A^{-1}\Jf_{n}\\
0&\Jf_{n}\,^t\!A^{-1}\Jf_{n}\end{smallmatrix}\right)\in R_1, &
r_2&\coloneqq j(A^{-1})\in R_2.
\end{align*}
A direct calculation shows that
\begin{align*}
r_1r_2\coloneqq\left(\begin{smallmatrix}
\id_n&-A\,^t\!A\Jf_{n}+\Jf_{n}\\
0&\id_n\end{smallmatrix}\right)&&r_2r_1\coloneqq\left(\begin{smallmatrix}
\id_n&-\Jf_{n}+A^{-1}\,^t\!A^{-1}\Jf_{n}\\
0&\id_n\end{smallmatrix}\right)
\end{align*}
so $r_1r_2=r_2r_1$ if and only if $2\id_n=A^{-1}\,^tA^{-1}+A\,^t\!A$.
Let us verify that such an equality never holds. Comparing the diagonal entries we obtain $a_i^2=(-1)^{i+1}a_0^{2i}(1-a_0^2)$ for $i>0$, 
whereas comparing the entries in the first row we obtain $a_1(a_2+a_0^3)=0$ and $a_1a_l=-a_0^3a_{l-1}$ for $l>2$. 
The conditions $a_2+a_0^3=0$ and $a_2^2=-a_0^4(1-a_0^2)$ lead to a contradiction, 
hence necessarily $a_1=0$ and so $a_l=0$ for any $l>0$ and $a_0^2=1$. In other words, $p_A(X)=X^n\pm1$. 
By Lemma \ref{lem:pol-irr}, this is possible only if $n=2$, $q\equiv3\mod 4$ and $p_A(X)=X^2+1$, which is excluded by hypothesis.

\medbreak
Then, $r_1\trid r_2\neq r_2$ and, for $H\coloneqq\langle r_1,r_2\rangle$ we have $\Oc_{r_1}^H\cap \Oc_{r_2}^H\subset R_1\cap R_2=\emptyset$ because $A^2\neq  \id$. 
If $p=2$, then $|r_1|=|r_2|$ is odd and $\Oc_{j(A)}^{\Sp_{2n}(q)}$ is of type C 
by Remark \ref{obs:equivC} $(b)$. If, instead, $p$ is odd, then $r_1r_2\neq r_2r_1$ 
implies $(r_1r_2)^2\neq(r_2r_1)^2$ as they are $p$-elements, so $\Oc_{j(A)}^{\Sp_{2n}(q)}$ is of type D.

\medbreak
We claim that the restriction of the projection $\pi\colon \Sp_{2n}(q)\to\Gb$ to $R_1\coprod R_2$ is injective. 
Indeed, this could fail only if $A^2=\pm1$, but since $A$ is irreducible, we would have $A^2=-1$ 
which would give $p_A(X)=X^2+1$, with $q\equiv3\mod 4$, i.e., the discarded case. 
Hence $\Oc_{\pi(j(A))}^{\Gb}$ collapses. 
\end{step}

\begin{step}
$\G=\SO_{2n}(\kk)$ or $\SO_{2n+1}(\kk)$. For $n\geq 3$ we consider the matrices:
\begin{align*}
E&\coloneqq
\begin{cases}
\diag (\id_{\frac{n}2},-\id_{\frac{n}2})&\textrm{if $n$ is even,}\\
\diag (\id_{[\frac{n}2]},0,-\id_{[\frac{n}2]})&\textrm{if $n$ is odd,}
\end{cases}\\
U&\coloneqq\begin{cases}
\left(\begin{smallmatrix}
\id_n&E\\
0&\id_n\end{smallmatrix}\right)&\textrm{if $\G=\SO_{2n}(\kk)$,}\\
\left(\begin{smallmatrix}
\id_n&0&E\\
0&1&0\\
0&0&\id_n\end{smallmatrix}\right)&\textrm{if $\G=\SO_{2n+1}(\kk)$.}
\end{cases}
\end{align*}
Then $U\in [\G^F,\G^F]$ by \cite[Theorem 24.15, Proposition 24.21]{MT} and   we consider the  elements $r_i\in R_i$,  $i=1,2$:
\begin{align*}
r_1&\coloneqq U\trid j(A)\in R_1=\begin{cases}
\left(\begin{smallmatrix}
A&-AE + E\phi(A)\\
0&\phi(A)\end{smallmatrix}\right)&\textrm{if $\G=\SO_{2n}(\kk)$,}\\
\left(\begin{smallmatrix}
A&0&-AE + E\phi(A)\\
0&1&0\\
0&0&\phi(A)\end{smallmatrix}\right)&\textrm{if $\G=\SO_{2n+1}(\kk)$,}
\end{cases}\\
r_2&\coloneqq j(A^{-1})\in R_2.
\end{align*}
A direct calculation shows that $r_1r_2=r_2r_1$ if and only if 
\begin{align}\label{eq:conditionD}2E = AE \phi(A^{-1})+A^{-1}E \phi(A).\end{align}
We verify that this never happens. Assume first that $p$ is odd. 
By looking at the $(1,1)$-entries we see that \eqref{eq:conditionD}  never holds if $n\geq 3$.  Since $r_1r_2$ and $r_2r_1$ are $p$-elements, it follows that $\pi(r_1r_2)^2\neq \pi(r_2r_1)^2$. The restriction of $\pi$ to $R_1\coprod R_2$ is injective because $A^2=-\id$ with $A$ irreducible would imply $n=2$, a discarded case. Hence $\Oc_{\pi(j(A))}^{\Gb}$ is of type D. 

\medbreak
Assume  that $p=2$, so $\G=\SO_{2n}(\kk)$. Then  \eqref{eq:conditionD} amounts to $A^2E \overset{\star}{=}
E\phi(A)^2$. If $n\geq 4$,  by looking at the first row we see that $\star$ never holds. 
If $n=3$, then  \eqref{eq:conditionD}  holds only when $a_2=a_0^{-1}$, $a_1=a_0^{2}$, 
but in this case 
\begin{equation*}
p_A(X)=X^3+a_0^{-1}X^2+a_0^2X+a_0=(X+a_0)^2(X+a_0^{-1})
\end{equation*} 
is not irreducible. 
Since $|r_1|$ is odd and $\pi$ is injective, Remark \ref{obs:equivC} \ref{item:equivC2} applies and so $\Oc_{\pi(j(A))}^{\Gb}$ is of type C. 
\qed
\end{step}

\begin{lema}\label{lema:sp-gl-irr} (Here $n\geq 3$ for $\G = \SO_{2n}(\kk)$ and $n\geq 2$ for $\G = \SO_{2n+1}(\kk)$ or $\Sp_{2n}(\kk)$).
Let $A\in \GL_n(q)$ be an irreducible element  such that  $j(A)$ is regular in $\GL_{n'}(q)$. 
Then
\begin{enumerate}[leftmargin=*,label=\rm{(\roman*)}] 
\item $\Os \coloneqq \Oc_{j(A)}^{[\G^F,\G^F]}$ collapses. 

\item\label{item:sp-gl-irr}   If $j(A) \in \Ot \subseteq [\G^F,\G^F]$, then $\Oc$ collapses. 
\end{enumerate}
\end{lema}
\begin{proof}
Since $A$ is irreducible, $\Oc_{A}^{\GL_n(q)}=\Oc_{A^q}^{\GL_n(q)}$.  
We have $\Os  =\Oc_{j(A^q)}^{[\G^F,\G^F]}$ by Lemma \ref{lema:properties-j} (i) 
 so we can consider the  disjoint subracks:
\begin{align*}
R_1&\coloneqq\left\{\left(\begin{smallmatrix}
A&Y\\
0&\phi(A)\end{smallmatrix}\right)\in\Os\right\},&
R_2&\coloneqq\left\{\left(\begin{smallmatrix}
A^{q}&Y
\\
0&\phi(A^q)\end{smallmatrix}\right)\in\Os\right\},  &\text{if } n'&=2n;
\\
R_1&\coloneqq\left\{\left(\begin{smallmatrix}
A&0&Y\\
0&1&0\\
0&0&\phi(A)\end{smallmatrix}\right)\in\Os \right\},&R_2&\coloneqq\left\{\left(\begin{smallmatrix}
A^{q}&0&Y\\
0&1&0\\
0&0&\phi(A^q)\end{smallmatrix}\right)\in\Os\right\},  &\text{if } n'&=2n+1.
\end{align*}
Then $R_i\trid R_j\subseteq R_j$ for $1\le i,j \le 2$. 

Let $r_1=j(A)\in R_1$. Since $j(A^q)$ is regular, $C_{\GL_{n'}(q)}(r_1)$ consists of semisimple elements,
 there exists  
$u\in [\G^F,\G^F]$ unipotent block upper triangular, with identity diagonal blocks of size $n,n$ if $n'=2n$ and $n,1,n$ if $n'=2n+1$,
such that $r_2\coloneqq u\trid j(A^q)\in R_2\setminus \{j(A^q)\}$.  
Observe that $r_2=j(A^q)v$ for some non-trivial block upper triangular unipotent element $v$. 
Now, $r_1 j(A^q)=j(A^q)r_1$ and $v \notin C_{\GL_{n'}(q)}(r_1)$ because the latter consists of semisimple elements. 
Hence, $r_1r_2 \neq r_2r_1$. 

\medbreak
If $p=2$, then $|A|$ 
is odd and $\Oc_{j(A)}^{[\G^F,\G^F]}$ is of type C by Remark \ref{obs:equivC} $(b)$. 
  
 \medbreak
 Let $p$ be odd.   
Then  $H\coloneqq\langle r_1,r_2\rangle=\langle r_1, v\rangle=\langle r_2, v\rangle$, with $v$ a $p$-element.  
Thus $\left|\Oc_{r_i}^H\right| \geq \left|\Oc_{r_i}^{\langle v\rangle}\right|\geq 3$ for $i=1, 2$,  
so $\Oc_{j(A)}^{[\G^F,\G^F]}$ is of type C by Lemma \ref{lem:equivC}.

\medbreak
Since $A$ is irreducible,  $A\neq A^q$,  hence the restriction of $\pi$ to $R_1\coprod R_2$ is injective,  
giving \ref{item:sp-gl-irr}. 
\end{proof}

\begin{lema}\label{lem:sp-gl-irr-exception}
Let $\G=\Sp_{4}(\kk)$, let $q\equiv 3\mod 4$ and let   $A=\left(\begin{smallmatrix}
0&-1\\
1&0\end{smallmatrix}\right)$. Then $\Oc=\Oc_{\pi (j(A))}^{\Gb}$ is of type D, hence it collapses.
\end{lema}

\begin{proof} By  assumption, $\pi( j(A))$ is an involution. 
Let $s_{c,d}\coloneqq\left(\begin{smallmatrix}
c&d\\
d&-c
\end{smallmatrix}\right) \in \F_q^{2\times 2}$, where $(c,d)\in \F_q^{2}$. 
A direct calculation shows that
\begin{align*}
A s_{c,d}A^{-1}=s_{-c,-d}=-s_{c,d}.
\end{align*} 
Thus, if $s_{c,d}\in\GL_2(q)$,  then 
\begin{align*}
\pi(j(A))\trid \pi(j(s_{c,d})) &=\pi(j(s_{c,d})) ;& &\text{also} & \phi(s_{c,d}) &= \frac{-1}{c^2+d^2}s_{c,-d}.
\end{align*}
We pick $(a,b)\in\F_q^2$ such that $a^2+b^2=-1$. Since $q\equiv3\mod4$, we have $ab\neq0$.
As $s_{a,b}$ is semisimple, 
with same trace and determinant as $A$, it lies in $\Oc_{A}^{\SL_2(q)}$, so  
 $\pi\left(j(s_{a,b})\right)\in \Oc_{\pi(j(A))}^{\Gb}$. 
 
Consider the  disjoint, non-empty subracks
\begin{align*}
R_1\coloneqq\left\{\pi\left(\begin{smallmatrix}
A&X\\
0&-A\end{smallmatrix}\right)\in\Oc_{\pi(j(A))}^{\Gb}\right\},&&R_2\coloneqq\left\{\pi\left(\begin{smallmatrix}
s_{a,b}&X\\
0&s_{a,-b}\end{smallmatrix}\right)\in\Oc_{\pi(j(A))}^{\Gb}\right\}
\end{align*}
of $\Oc_{\pi(j(A))}^{\Gb}$. 
Then $R_i\trid R_j=R_j$ for $i,j \in \{1,2\}$. 

We set 
\begin{align*}
r&\coloneqq\pi\left(\begin{smallmatrix}
\id_2&\id_2\\
0&\id_2
\end{smallmatrix}\right)\trid \pi(j(A))=\pi\left(\begin{smallmatrix}
A&-2A\\
0&-A
\end{smallmatrix}\right)\in R_1, &
s&\coloneqq\pi(j(s_{a,b}))\in R_2. 
\end{align*}
Now $ab\neq0$ implies that $s_{a,b}s_{a,-b}$ is not diagonal, hence 
\begin{align*}(rs)^2=\pi\left(\begin{smallmatrix}
\id_2&2(\id_2+ s_{a,b}s_{a,-b})\\
0&\id_2\end{smallmatrix}\right)\not= \pi\left(\begin{smallmatrix}
\id_2&-2(\id_2 + s_{a,b}s_{a,-b})\\
0&\id_2\end{smallmatrix}\right)=(sr)^2,\end{align*} so $\Oc_{\pi(j(A))}^{\Gb}$ is of type D.  
\end{proof}

\subsection{Proofs of Theorems \texorpdfstring{\ref{thm:section-meetsK-PSp}}{} and \texorpdfstring{\ref{thm:section-meetsK-PSO}}{} } \label{pfthm:section-meetsK}
We now drop the  irreducibility assumption and proceed to prove the main results of this Section.

\medbreak
\noindent\emph{Proof of Theorem \ref{thm:section-meetsK-PSp}}. 
 For $A$  irreducible, this is  Lemmata \ref{lema:spso-gl-irr}, \ref{lema:sp-gl-irr} 
and  \ref{lem:sp-gl-irr-exception}. 

\medbreak
If $A$ is not irreducible, then we may assume that $A$ is a block diagonal matrix $\diag (A_1,\cdots,A_f)$ 
where the $A_i$'s are irreducible. 
If they are all of size $1$, then $j(A)$ lies in a $\F_q$-split torus and  Proposition
\ref{thm:split-collapses} applies.
If, instead, one of the matrices $A_i$  has size $n_i\geq 2$, then $n>2$  and $A_i$ is non-central in $\GL_{n_i}(q)$ 
because it is irreducible. 
Lemmata  \ref{lema:spso-gl-irr}, \ref{lema:sp-gl-irr} and  \ref{lem:sp-gl-irr-exception} imply that 
$\Oc_{j(A_i)}^{\Sp_{2n_i}(q)}$ collapses.  
The statement follows from injectivity of the composition of rack maps: 
 \begin{align*}
 &\Bigg(\prod_{l=1}^{i-1}\{j(A_l)\}\Bigg)\times 
 \Oc_{j(A_i)}^{\Sp_{2n_i}(q)}\times \Bigg(\prod_{m=i+1}^{f}\{j(A_m)\}\Bigg)
 \to \Oc_{j(A)}^{\Sp_{2n}(q)}\to \Oc_{\pi (j(A))}^{\Gb}. \hspace{-2pt}\qed
 \end{align*}

\bigbreak
\noindent\emph{Proof of Theorem \ref{thm:section-meetsK-PSO}}. 
For $A$  irreducible, this is  Lemmata \ref{lema:spso-gl-irr} and \ref{lema:sp-gl-irr}. 

\medbreak 
If $A$ is not irreducible, then we may assume that $A$ is a block diagonal matrix $\diag (A_1,\cdots,A_f)$ where $f>1$
and each $A_i$ is an irreducible $n_i\times n_i$-matrix.   

\medbreak
If $q=2$, then  $A$ lies in $\SL_n(q)$ and  is not irreducible, 
so the rack inclusion $\Oc_{A}^{\SL_n(q)} \hookrightarrow \Oc_{j(A)}^{[\SO_{n'}(q), \SO_{n'}(q)]}$ 
combined with \cite[Theorem 1.1]{ACG-III} gives the claim because $\Pom^+_{n'}(q)=[\SO_{n'}(q),\SO_{n'}(q)]$. 

\medbreak
If $n_i=1$ for all $i$, then $j(A)$ lies in a $\F_q$-split torus and   Lemma \ref{lema:SO-split} applies.  

\medbreak
Therefore,\emph{ we  assume from now on that  $q>2$ and  $n_i\geq 2$ for some $i$}.

\medbreak
If $n_i\geq 3$ for some $i$,  then $\Oc_{j(A_i)}^{[\SO_{2n_i}(q),\SO_{2n_i}(q)]}$ collapses by
Lemmata \ref{lema:spso-gl-irr} and \ref{lema:sp-gl-irr}.
Then the  claim follows because of the injectivity of the composition of the
rack morphisms
\begin{align*}
 \left(\prod_{l=1}^{i-1}\{j(A_l)\}\right)\times \Oc_{j(A_i)}^{[\SO_{2n_i}(q), 
 \SO_{2n_i}(q)]}\times\left(\prod_{l=i+1}^{f}\{j(A_l)\}\right)&\to \Oc_{j(A)}^{[\SO_{n'}(q),\SO_{n'}(q)]}\\
 &\to \Oc_{\pi(j(A))}^{\Gb}.
 \end{align*}

\emph{From now on we assume that   $n_1=2$, and  $n_i\leq 2$ for all $i$.  }
  
 \medbreak
If $n_i=1$ for some $i$, say  $i=2$, then $A_2=(c)$ for some $c\in\F_q^{\times}$.  Since $A_1$ is irreducible, it is regular and has no eigenvalues in $\F_q$. Thus the block diagonal matrix $\tilde{A}_1=\diag (A_1,c)$ has $3$ distinct eigenvalues in $\overline{\F_q}$. The matrices  
$\left(\begin{smallmatrix}
A_1&v\\
0&c\end{smallmatrix}\right)$, $v\in\F_q^2$,
have the same eigenvalues, hence  they  lie in $\Oc_{\tilde{A_1}}^{\GL_3(q)} = \Oc_{\tilde{A_1}}^{\SL_3(q)}$, cf.
Remark \ref{obs:normal-subrack}. 
Consider the map $j: \GL_3(q) \to \SO_{6}(q)$.
We claim that $j\left(\begin{smallmatrix}
A_1^{-1}&0\\
0&c\end{smallmatrix}\right)\in \Os \coloneqq\Oc_{j(\tilde{A}_1)}^{[\SO_6(q),\SO_6(q)]}$.

Indeed, there is a representative $g$ of a suitable $w \in W$ in the normaliser of the torus of diagonal matrices in $[\SO_{6}(q),\SO_6(q)]$ that satisfies
$g\trid  j(\tilde{A_1}) = j\left(\begin{smallmatrix}
^tA_1^{-1}&0\\
0&c\end{smallmatrix}\right)$.
Also, $^tA_1^{-1} \in \Oc_{A^{-1}}^{\SL_3(q)}$, hence $j\left(\begin{smallmatrix}
^tA_1^{-1}&0\\
0&c\end{smallmatrix}\right) \in \Os$.
Thus
\begin{align*}R_1&=\left\{j\left(\begin{smallmatrix}
A_1&v\\
0&c\end{smallmatrix}\right): v\in\F_q^2\right\}, && R_2=\left\{j\left(\begin{smallmatrix}
A^{-1}_1&v\\
0&c\end{smallmatrix}\right): v\in\F_q^2\right\}\end{align*}
are subracks of $\Os$, which are disjoint because $A_1$ is irreducible. Clearly,  $R_i\trid R_j\subset R_j$, $1 \le i,j \le 2$.  Pick $0 \neq v\in\F_q^2$ and set:
\begin{align*} 
&r\coloneqq j\left(\begin{smallmatrix}
A_1&0\\
0&c\end{smallmatrix}\right)\in R_1;&&s\coloneqq j\left(\begin{smallmatrix}
A^{-1}_1&v\\
0&c\end{smallmatrix}\right)\in R_2.\end{align*}
By a direct calculation,  $rs=sr$  implies that $c$ is an eigenvalue of $A_1$, a contradiction. 
Similarly, $(rs)^2=(sr)^2$ iff $c^2=-1$, which can occur only if $q$ is even or $q\equiv 1 \hspace{-5pt}\mod 4 $. 
If $q$ is even, then $\Os$ is of type C by  Remark \ref{obs:equivC} $(b)$. 
If $q\equiv 3  \hspace{-5pt}\mod 4 $ or $q\equiv 1 \hspace{-5pt} \mod 4 $ and $c^2\neq -1$, then  $\Os$ is of type D.
 
 \medbreak
 Assume that $q\equiv 1 \mod 4 $ and $c^2= -1$.  We claim that
 \begin{align} \label{eq:inequality1}
 3=|\{s,rsr^{-1},  r^2 sr^{-2}\}|&\leq\left|\Oc_s^{\langle r,s\rangle}\right|, 
 \\\label{eq:inequality2}
 3= |\{r,srs^{-1},  rsrs^{-1}r^{-1} \}| &\leq\left|\Oc_r^{\langle r,s\rangle}\right|
 \end{align}
 By a direct calculation, $r^2 sr^{-2}=s$ iff 
 $A_1^2v = -v = c^2v$, that is,  hence $c$ or $-c$ is an eigenvalue of $A_1$, a contradiction because they both lie in $\F_q$; 
 \eqref{eq:inequality1} follows.
 Similarly, $rsrs^{-1}r^{-1}= srs^{-1}$ iff $(A_1 - c)^2 v = 0$, 
 thus $c$  is an eigenvalue of $A_1$, a contradiction. 
 Now
 $srs^{-1}\neq r$ implies $rsrs^{-1}r^{-1}\neq r$ and \eqref{eq:inequality2} follows.
 Hence $\Oc_{j(\tilde{A}_1)}^{[\SO_3(q),\SO_3(q)]}$ is of type C by Lemma \ref{lem:equivC}. Since the composition
 \begin{align*}(R_1\coprod R_2)\times  \left(\prod_{l=3}^{f}\{j(A_l)\}\right)&\to  
 \Oc_{j(\tilde{A}_1)}^{[\SO_3(q),\SO_3(q)]}\times \left(\prod_{l=3}^{f}\{j(A_l)\}\right)\\
 &\to \Oc_{j(A)}^{[\SO_{n'}(q),\SO_{n'}(q)]} \to \Oc =\Oc_{\pi (j(A))}^{\Gb}\end{align*} 
 is an injective  morphism of racks,  the statement is proved in this case. 
 
 \medbreak
There remains the case $n_i=2$ for all $i$. It suffices to assume that $f=2$, so $\G=\SO_8(q)$, and that $A_1$ and $A_2$ are the companion matrices of their  characteristic polynomials $p_{A_{1}} = X^2+aX+b$ and $p_{A_{2}} = X^2+cX+d$, so
\begin{align*}
&A_1\coloneqq\left(\begin{smallmatrix}
0&-b\\
1&-a\end{smallmatrix}\right), &A_2\coloneqq\left(\begin{smallmatrix}
0&-d\\
1&-c\end{smallmatrix}\right),&&A\coloneqq\left(\begin{smallmatrix}
A_1&0\\
0&A_2\end{smallmatrix}\right).
\end{align*}
As in the previous step, there is an element in $[\SO_8(q),\SO_8(q)]$ mapping $j(A)$ to $j\left(\begin{smallmatrix}
A^{-1}_1&0\\
0&A_2\end{smallmatrix}\right)$.
We consider the subracks of $\Oc_{j(A)}^{[\SO_8(q),\SO_8(q)]}$ given by
\begin{align*}
&R_1\coloneqq\left\{ j\left(\begin{smallmatrix}
A_1&M\\
0&A_2\end{smallmatrix}\right)\in \Oc^{[\SO_8(q),\SO_8(q)]}_{j(A)}\right\},\\
&R_2\coloneqq\left\{ j\left(\begin{smallmatrix}
A^{-1}_1&M\\
0&A_2\end{smallmatrix}\right)\in \Oc^{[\SO_8(q),\SO_8(q)]}_{j(A)}\right\}
\end{align*}
which are disjoint since $A_1$ is irreducible. Clearly, $R_i\trid R_j\subset R_j$, $1\le i,j \le 2$. 
Let $u\coloneqq\left(\begin{smallmatrix}
\id_2&\id_2\\
0&\id_2\end{smallmatrix}\right)\in\SL_4(q)$ and consider
\begin{align*}
r &\coloneqq j(u)\trid j(A)=j\left(\begin{smallmatrix}
A_1&A_2-A_1\\
0&A_2
\end{smallmatrix}\right)\in R_1, &
s &\coloneqq j\left(\begin{smallmatrix}
A^{-1}_1&0\\
0&A_2
\end{smallmatrix}\right)\in  R_2.
\end{align*}
A direct calculation in $\GL_4(q)$ shows that $(rs)^2= (sr)^2$ if and only if 
\begin{align}\label{eq:2-2}(A_2-A_1)A_2(\id_2+A_2^2)=A_1^{-1}(A_2-A_1)(\id_2+A_2^2).\end{align}
Now, $\det(\id_2+A_2^2)=0$ implies that there exists  $0 \neq v\in\F_q^2$ such that $A_2^2v=-v$. By the 
irreducibility of $A_2$, we  get $p_{A_2} = X^2 + 1$, i.e., $A_2 = \left(\begin{smallmatrix}
0&-1\\
1&0 \end{smallmatrix}\right)$. 

\medbreak
Assume  that $\det(\id_2+A_2^2)\neq0$. Then \eqref{eq:2-2} is equivalent to  $A_2-A_1=A_1^{-1}-A_2^{-1}$, 
which is equivalent to $A_1=A_2$. 
Therefore, if $A_1\neq A_2$, possibly interchanging the roles of $A_1$ and $A_2$, we can make sure that \eqref{eq:2-2} is not satisfied, hence $(rs)^2\neq(sr)^2$, so $\Oc_{A}^{\SL_4(q)}$ is of type $D$. 
If $A_1=A_2\neq \left(\begin{smallmatrix}0&-1\\
1&0\end{smallmatrix}\right)$, then we  interchange $A_1$ and $A_1^{-1}$ and argue as above. 
In all cases, the restriction of $\pi$ to $R_1\coprod R_2$ is injective, so $\Oc_{\pi(j(A))}^{\Gb}$ is  of type D.

Finally, if  $A_1=A_2= \left(\begin{smallmatrix}0&-1\\
1&0\end{smallmatrix}\right)$, then $A\in\SL_4(q)$ and $\Oc_{\pi_{\SL_4(q)}(A)}^{\PSL_4(q)}$ is of type D by \cite[Lemmata 3.15, 3.16, 3.17]{ACG-III}, and  $\Oc_{\pi (j(A))}^{\pi(K)}$ projects onto it. Whence  $\Oc_{\pi (j(A))}^{\pi(K)}$  is also of type D.
\qed

\section{The symplectic groups}\label{sec:symplectic}

In this Section, $\G = \Sp_{2n}(\kk)$, $n \geq 2$. Recall that $e \neq x \in \Gb$ is semisimple, 
$\G^F \ni \x \mapsto x$, $\Ot = \Oc_{\x}^{\G^F}$ and $\Oc = \Oc_x^{\Gb}$.
Here is the main result of this Section:

\begin{theorem}\label{thm:conclusion-semisimple-SP}
Let $\x\not\in Z(\G)$ be semisimple. 
Then $\Oc$ collapses unless  $n=2$, $q\in \{3,5,7\}$ and $\x$ is an involution. 
\end{theorem} 

 Classes represented in $K$ have been discussed in Section \ref{sec:K}. We deal in Subsection \ref{subsec:cuspidal-sp2n} with cuspidal classes that are not Coxeter, and then with Coxeter classes in Subsection \ref{subsec:coxeter}. 
 Theorem \ref{thm:conclusion-semisimple-SP} is proved in Subsection \ref{subsec:symplectic-non-cuspidal}.

\subsection{Cuspidal classes}\label{subsec:cuspidal-sp2n}
Here we discuss the semisimple classes that are cuspidal but not Coxeter.
Below we use without further notice that a cuspidal class could not meet a standard Levi subgroup by Proposition
\ref{prop:cuspidal-classes}. 
We start by the following observation: 
two semisimple symplectic matrices conjugated in $\GL_{2n}(q)$ are then  conjugated in $\Sp_{2n}(q)$.

\begin{lema}\label{lema:sp-cuspidal-loworder-eigenvalues} In either of the following cases, $\Ot$ is not cuspidal:
\begin{enumerate*}[leftmargin=*,label=\rm{(\alph*)}]
\item\label{item:sp-eigenvalue}
Some eigenvalue of $\x$ lies in $\F_q$.

\item\label{item:sp-loworder} $\vert x \vert \in \{2,3,4\}$.
\end{enumerate*}
\end{lema}

\pf \ref{item:sp-eigenvalue}. Indeed, if $\lambda \in \F_q$ is an eigenvalue of $\x$, then so is 
$\lambda^{-1}$, hence $\Ot$ contains an element of the form
$\left(\begin{smallmatrix} \lambda & & &\\  &A' & B' &  \\ &C' & D' & \\ &&& \lambda^{-1} \end{smallmatrix}\right)$
which belongs to a Levi subgroup isomorphic to $\Sp_{2(n-1)}(\kk) \times \kk^{\times}$. Thus $\Ot$ is not cuspidal.

\ref{item:sp-loworder}.
By \ref{item:sp-eigenvalue}, we may assume that $\x$ has no eigenvalues in $\F_q$, so $\pm1$ are excluded. If $|\x|\in\{2,3,4\}$, then $\x$ has at most $2$ distinct eigenvalues, namely the two primitive roots of $1$, so it is not cuspidal by Proposition \ref{prop:one-conj-class-regular-bis}.
\epf

\subsubsection{Cuspidal classes in the Weyl group}\label{subsubsec:cuspidal-sp2n-weyl}
As is well-known, the Weyl group is $W = (\Z/2)^n \rtimes \sn$; let $(e_j)_{j\in \I_n}$ be the canonical basis of $(\Z/2)^n$.  
We identify $W$ with a subgroup of $\s_{2n}$ as in \cite{BL}:
\begin{align*}
W &\simeq \{\varsigma \in \s_{2n}: \varsigma(2n+1-j) = 2n+1 - \varsigma(j), j\in \I_{n} \},
\\
\sn &\ni \sigma \mapsto \sigma', \qquad \sigma'(j) = \begin{cases}
\sigma(j), &\text{if } j\in \I_{n}, \\ 
2n+1 - \varsigma(2n+1 -j), &\text{if } j\in \I_{n+1,2n}; 
\end{cases} \\
e_j & \mapsto \tau_j = (j \quad 2n+1 - j), \qquad j\in \I_{n}.
\end{align*}

Given $h \le k$ in $\I_n$, we consider the $2(k-h + 1)$-cycle in $\s_{2n}$ defined by
\begin{align}\label{eq:Chk}
\ct_{h,k} &= (h\quad h+1 \ \dots \  k \quad 2n +1- h \quad 2n - h \ \dots \  2n +1-k).
\end{align} 
Evidently, $\ct_{h,k}  \in W$. Let now $\blambda = (\letra_1, \dots, \letra_t)$ be a partition of $n$, 
denoted $\blambda \vdash n$, with 
$\letra_1 \geq\dots \geq \letra_t$. Set
\begin{align}\label{eq:Cblambda}
\ct_{\blambda} &= \ct_{1, \letra_1}\ct_{\letra_1 + 1, \letra_1 + \letra_2} \dots \ct_{\letra_1+ \dots + \letra_{t-1} + 1, n} \in W.
\end{align} 

By the identification above, we can rephrase \cite[Proposition 3.4.6]{GP}:

\begin{prop}\label{prop:cuspidal-sp}  The conjugacy class of such $\ct_{\blambda}$ is cuspidal. 
The family $\ct_{\blambda}$, $\blambda \vdash n$, is a complete set of representatives
of the cuspidal conjugacy classes of $W$. \qed
\end{prop}

For instance, if $\blambda = (n)$, then  $\ct_{\blambda}$ is a Coxeter element.

\subsubsection{Cuspidal, but not Coxeter, classes}\label{subsubsec:cuspidal-sp2n-cusp}
 Let  $\blambda = (\letra_1, \dots, \letra_t) \vdash n$, with 
$\letra_1 \geq\dots \geq \letra_t$. Let
$\G_{\blambda}$ be the image of the injective morphism of groups

\begin{align*}
\Sp_{2\letra_1}(\kk) \times  \Sp_{2\letra_2}(\kk) \times \dots \times \Sp_{2\letra_t}(\kk) &\longrightarrow \Sp_{2n}(\kk), 
\end{align*}
\begin{align*}
\left(\begin{pmatrix}
A_1 & B_1 \\ C_1 & D_1\end{pmatrix}, \dots, \begin{pmatrix}
A_t & B_t \\ C_t & D_t\end{pmatrix}\right) &\longmapsto
\begin{pmatrix}
A_1 & & & & & B_1 
\\
& \ddots & & & \reflectbox{$\ddots$}  &
\\
& &A_t & B_t& & \\& & C_t & D_t& &
\\
& \reflectbox{$\ddots$} & & & \ddots  &
\\ C_1 & & & & & D_1
\end{pmatrix}.
\end{align*}

\noindent\textbf{Claim.}
$\ct_{\blambda}$ has a decomposition $\varGamma$  such that
$\G_{\varGamma} = \G_{\blambda}$, cf. \eqref{eq:subgrp-decomposition}.

\pf First, $w_j = \ct_{\letra_1 + \letra_2 + \dots + \letra_{j-1}+ 1, \letra_1 + \letra_2 + \dots + \letra_{j}}$
is a Coxeter element of the factor $\Sp_{2\letra_j}(\kk)$ of $\G_{\blambda}$.
Up to appropriate identifications, the union of decompositions $\varGamma_1$, \dots,
$\varGamma_t$ of $w_1$, \dots, $w_t$ is a decomposition of $\ct_{\blambda}$.
This implies the claim.
\epf

\begin{lema}\label{lema:sp-cuspidal-collapses} 
If the conjugacy class $\Ot$ in $\G^F$ is cuspidal but not Coxeter, then it is of type C, hence it collapses. 
\end{lema}

\pf By Proposition \ref{prop:cuspidal-classes}, there is partition $\blambda \neq (n)$ such that
$\Ot$ intersects $\G_{\blambda}^F = \Sp_{2\letra_1}(q) \times  \Sp_{2\letra_2}(q) \times \dots \times \Sp_{2\letra_t}(q)$.
Let $\x = (\x_1, \dots, \x_t) \in \G_{\blambda}^F \cap \Ot$, with $\x_j \in \Sp_{2\letra_j}(q)$ for all $j$, so that
\begin{align*}
\Oc_{\x}^{\G_{\blambda}^F} =  \Oc_{\x_1}^{\Sp_{2\letra_1}(q)} \times \Oc_{\x_2}^{\Sp_{2\letra_2}(q)} \times \dots \times \Oc_{\x_t}^{\Sp_{2\letra_t}(q)} \leq \Ot.
\end{align*}
We claim that $\x_j \notin Z(\Sp_{2\letra_j}(q))$ for all $j \in \I_t$. Indeed, if
$\x_j \in Z(\Sp_{2\letra_j}(q))$ for some $j$, then $\x$ belongs to the torus $\T_{\ct_{\blambda}^i}$, cf. \eqref{eq:Tw},
where $\ct_{\blambda}^j \in W$  is defined as $\ct_{\blambda}$ in \eqref{eq:Cblambda} but omitting 
$\ct_{\letra_1+ \dots + \letra_{j-1} + 1, \letra_1+ \dots + \letra_{j}}$. This is a contradiction because 
$\ct_{\blambda}^j$ is not cuspidal proving the claim. 
Then the lemma follows from  Lemma \ref{lem:typeC-directproduct}, by Remark \ref{obs:typeC-directproduct} and 
Lemma \ref{lema:x-xq}. Indeed, $\x_j \neq \x_j^q$, otherwise $\x$ would lie in a non-cuspidal torus by Lemma \ref{lema:sp-cuspidal-loworder-eigenvalues}
\epf

\begin{lema}\label{lema:psp-cuspidal-collapses} 
If the conjugacy class $\Ot$ in $\G^F$ is cuspidal but not Coxeter, then 
the conjugacy class $\Oc$ in $\Gb$ is of type C, hence it collapses. 
\end{lema}

\pf Let $\pi:\Sp_{2n}(q) \to \PSp_{2n}(q)$ be the canonical projection.
We may assume that $q$ is odd, thus $\ker \pi = \{\pm \id\}$. 
Keep the notation of  Lemma \ref{lema:sp-cuspidal-collapses}.

\begin{claimpsp} The Lemma holds for $t > 2$.
\end{claimpsp}

Indeed, the Lemma \ref{lem:typeC-directproduct}  provides a subrack of $\Ot$ of type C with the form 
$$Y = \left(\{\x_1\} \times \Oc_{\x_2}^{\Sp_{2\letra_2 }(q)} \times \{\x_3\} \right)
\coprod \left(\{\x_1^q\} \times \Oc_{\x_2}^{\Sp_{2\letra_2 }(q)} \times \{\x_3\}\right)$$
and clearly the restriction of $\pi$ to $Y$ is injective.

\begin{claimpsp} The Lemma holds for $t = 2$.
\end{claimpsp}

If $\x_1 \neq  -\x_1^q$, then the restriction of $\pi$ to the subrack of type C
$$Y = \left(\{\x_1\} \times \Oc_{\x_2}^{\Sp_{2\letra_2 }(q)}\right)
\coprod \left(\{\x_1^q\} \times \Oc_{\x_2}^{\Sp_{2\letra_2 }(q)}\right)$$
is injective; similarly if $\x_2 \neq  -\x_2^q$. Thus we may assume that $\x_1 = -\x_1^q$ and 
$\x_2 = -\x_2^q$.
Now $\x_1$ lives in a Coxeter torus $\T_1^F$ in $\Sp_{2\letra_1}(q)$. By \cite[Proposition 25.3]{MT} and 
\cite[\S 3.4.3]{GP}, we have
\begin{align*}
\vert \T_1^F \vert &= q^{\letra_1} + 1.
\end{align*}
Hence $\vert \x_1\vert$ divides $(2(q - 1), q^{\letra_1} + 1) = \begin{cases}
2 &\text{if } q^{\letra_1} \equiv 1 \mod 4, \\
4 &\text{if } q^{\letra_1} \equiv 3 \mod 4.
\end{cases}$

By symmetry, we may assume that the same holds for $\x_2$. Hence $\vert \x\vert$ divides 4;
this contradicts Lemma \ref{lema:sp-cuspidal-loworder-eigenvalues}.
\epf

\subsection{Coxeter classes in \texorpdfstring{$\Sp_{2n}(q)$}{}}\label{subsec:coxeter}

Let $\x\in T'=\T_w^F$ be a Coxeter element  and let $\Ot=\Oc_{\x}^{\G^F}$. 
Hence $\x$ is regular and its order divides $q^n+1$, so $\x^{q^n}=\x^{-1}$. Arguing as in \cite[\S 2.5]{ACG-V} we see that $\Ot\cap T'=\{\x^{\pm q^j},\;j\in\I_{0,n-1}\}$, and that the action of 
$w$ raises $\x$ to $\x^q$. If $\xi\in\overline{\F_q}$ is an eigenvalue of $\x$, then all other eigenvalues of $\x$ are of the form $\{\xi^{q^j},\; j\in\I_{0,2n-1}\}=\{\xi^{\pm q^j},\;j\in\I_{0,n-1}\}$, with $\xi^{q^n}=\xi^{-1}$ and they are all distinct by Proposition \ref{prop:one-conj-class-regular-bis}.

\begin{lema}\label{lem:-xnotinO}Assume $q$ is odd. If $\x$ is a Coxeter element, then $-\x\not\in\Ot$.
\end{lema}
\begin{proof}If $-\x\in\Ot$, then with notation as above,  
$-\xi$ is an eigenvalue of $\x$, so $-\xi=\xi^{q^j}$ or $-\xi=\xi^{-q^j}$ for some $j<n$. 
In the first case $\xi^{q^{2j}}=(-\xi)^{q^j}=\xi$,  whilst in the second  
$\xi^{q^{2j}}=(-\xi^{-1})^{q^j}=-(\xi^{q^j})^{-1}=\xi$, with $2j<2n$ in both cases, contradicting regularity of $\x$. 
\end{proof}

\begin{lema}\label{lema:sp-coxeter}
Let $\x$ be a Coxeter element in $\G^F$. Then $\Oc$ is of type C.
\end{lema}
\begin{proof}Let $H$ be a subgroup  of $\G^F$  isomorphic to $\SL_2(q^n)$, 
which exists  by \cite[II  Satz 9.24]{Huppert}.  Any non-split torus of $T'\leq H$ 
has order $q^n+1$ by \eqref{eq:gln-coxeter-torus}, hence it is a Coxeter torus in $\G^F$, \cite[3.4.3]{GP}. Therefore 
\begin{align*}|\Ot\cap T'|=|\{\x^{\pm q^j},\;j\in\I_{0,n-1}\}|=2n,&&|\Oc_{\x}^H\cap T'|\leq2,\end{align*}
so the intersection $\Ot\cap H$ is not a single $H$-conjugacy class. Assume $\x\in H$. 
Since  $H\simeq \SL_2(q^n)$ with $n\geq 2$, the group $H/Z(H)$ is simple, so the non-central normal 
subgroup $\langle \Oc_{\x}^H\rangle$ coincides with $H$.  
In addition, $|\Oc_{\x}^H|=q^n(q^n-1)>4$, so  $\Ot$ is of type C by Lemma  \ref{lem:typeC-subgroup}. 
The restriction of the projection $\pi\colon \G^F\to\Gb$ to $\Ot$ is injective by Lemma \ref{lem:-xnotinO}, 
so $\pi(\Ot)=\Oc$ is of type C as well.
\end{proof}

\subsection{The general case}\label{subsec:symplectic-non-cuspidal}
Let $\Le$ be a split $F$-stable Levi subgroup of $\G$. Then, there exist $f>0$,  $m\geq0$ and $n_i$ for $i\in \I_f$ satisfying $n=e+\sum_{i=1}^fn_i$ such that 
$\Le$ is isomorphic the image of the injective morphism of groups 
\begin{align}\label{eq:morfismo}
&\widetilde{j}\colon \GL_{n_1}(\kk)\times\cdots\times\GL_{n_f}(\kk)\times\Sp_{2e}(\kk) \to \Sp_{2n}(\kk)\\
& (A_1,\cdots,A_f, A) \mapsto \diag(A_1,\ldots,\,A_r,\,A,\,\phi(A_f),\,\ldots,\,\phi(A_1))
\end{align}

\medbreak
\noindent\emph{Proof of Theorem \ref{thm:conclusion-semisimple-SP}}. 
If $n=2$,  $q=3$  and $\x$ is  a non-central involution, then $\Oc$ is kthulhu by Lemma \ref{lem:special-q-3}. 
Assume that $q\notin \{5,7\}$ if $n=2$ and $\x$ is  a non-central involution.
If $\x$ is cuspidal but not Coxeter,  we invoke Lemma \ref{lema:psp-cuspidal-collapses}, whilst if $\x$ is Coxeter, then the claim follows from Lemma \ref{lema:sp-coxeter}. 
If $\x$ is not cuspidal, then by Proposition \ref{prop:cuspidal-classes} we may assume that $\x\in\Le^F$ for a proper standard Levi subgroup $\Le$ of $\G$. 
Let $\widetilde{j}$ be as in \eqref{eq:morfismo} and let $\x=\widetilde{j}(\x_1,\ldots,\x_f,\y)$. 
Taking $n_i$, for $i\in \I_f$ and $e$ to be minimal, and possibly increasing $f$, we assume that each $\x_i$ is irreducible in $\GL_{n_i}(q)$ and $\y$ is cuspidal in $\Sp_{2e}(q)$.
Under these assumptions $\x_i\in Z(\GL_{n_i}(q))$ if and only if $n_i=1$.
If $e=0$ the statement follows from Theorem \ref{thm:section-meetsK-PSp}. 
If $e\geq 2$, then we consider the rack embedding
\begin{align}
\{\x_1\}\times\cdots\times\{\x_f\}\times\Oc_{\y}^{\Sp_{2e}(q)}\to \Oc_{\x}^{\Sp_{2n}(q)}\to \Oc_x^{\PSp_{2n}(q)}
\end{align}
and invoke either  Lemma \ref{lema:psp-cuspidal-collapses} or Lemma \ref{lema:sp-coxeter}. 

Assume from now on that  $e=1$, i.e., $\y$ is irreducible in $\Sp_{2}(q)\simeq \SL_2(q)$.
If there exists and $n_i$ such that $n_i>1$, then 
we consider the rack embedding
\begin{align}
\{\x_1\}\times\cdots\times\Oc_{\x_i}^{\Sp_{2n_i}(q)}\times\cdots\times\{\x_f\}\times \{\y\}\to \Oc_{\x}^{\Sp_{2n}(q)}\to \Oc_x^{\PSp_{2n}(q)}
\end{align}
and invoke Theorem \ref{thm:section-meetsK-PSp}. 

There remains the case in which $n_i=1$ for every $i$. We assume that $f=1$, for if $f>1$ we can use the rack injection
\begin{align*}
\{\x_1\}\times\cdots\times\{\x_{f-1}\}\times\Oc_{\tilde{j}(\x_f,\y)}^{\Sp_{2(n_f+1)}(q)}\to \Oc_{\x}^{\Sp_{2n}(q)}\to \Oc_x^{\PSp_{2n}(q)}.
\end{align*}
Since $\y$ is irreducible, it lies in a non-split maximal torus, so its order divides $q+1$ and so $\y^q=\y^{-1}\in \Oc_{\y}^{\SL_2(q)}$. Also, if $p_{\y} = X^2-z X+1$, then $z\neq\pm2$.  We may assume that 
$\y=\left(\begin{smallmatrix}
0&1\\
-1&z\\
\end{smallmatrix}\right)$ so $\y^{-1}=\left(\begin{smallmatrix}
z&-1\\
1&0\\
\end{smallmatrix}\right)$ and  $\x=\diag (\lambda,\y,\lambda^{-1})$.
We consider the following subracks of $\Ot = \Oc_{\x}^{\Sp_{4}(q)}$:
\begin{align*}
R&\coloneqq\left\{\x'= \left(\begin{smallmatrix}
\lambda&*&*\\
0&\y&*\\
0&0&\lambda^{-1}\end{smallmatrix}\right): \x' \in \Ot \right\},
&
 S&\coloneqq\left\{\x'= \left(\begin{smallmatrix}
\lambda^{-1}&*&*\\
0&\y^{-1}&*\\
0&0&\lambda\end{smallmatrix}\right): \x' \in\Ot \right\}.
\end{align*}
By construction, $R\trid S\subset S$ and $S\trid R\subset R$. Observe that $R\cap S=\emptyset$; otherwise $\y=\y^{-1}$ and $p=2$, but in this case $\y$ would not be semisimple.
Let $M\in \SL_2(q)$ be such that $M\trid \y=\y^{-1}$ and let
\begin{align*}
r&\coloneqq\left(\begin{smallmatrix}
1&1&0&0\\
0&1&0&0\\
0&0&1&-1\\
0&0&0&1
\end{smallmatrix}\right)\trid \x=
\left(\begin{smallmatrix}
\lambda&-\lambda&1&1\\
0&0&1&1\\
0&-1&z&z-\lambda^{-1}\\
0&0&0&\lambda^{-1}\end{smallmatrix}\right)\in R, \\
s&\coloneqq \left(\begin{smallmatrix}
0&0&1\\
0&M&0\\
-1&0&0
\end{smallmatrix}\right)\trid \x=\diag (\lambda^{-1},\y^{-1},\lambda)\in S.
\end{align*}
A direct calculation shows that $rs=sr$ only if $\lambda^2=1$ and $z=\pm2$, a discarded case.  
Taking $H\coloneqq\langle r,\,s\rangle$, we see that $\Oc_{r}^H\cap \Oc_s^H\subset R\cap S=\emptyset$. Thus, if  $p=2$, then $|\x|$ is odd so  $\Oc_{\x}^{\Sp_4(q)}$ is of type C  by Remark \ref{obs:equivC}. 
If $p$ is odd, then $(rs)^2\neq(sr)^2$ because $rs$ and $sr$ are upper triangular unipotent matrices by construction, and so $\Oc_{\x}^{\Sp_{4}(q)}$ is of type D.
We  claim that the restriction of $\pi$ to $R\coprod S$ is injective: indeed injectivity could fail only for $\lambda^2+1=0$ and $z=0$ but in this case, $\lambda\in\F_q$ would be a root of $p_\y$ which is irreducible, a contradiction.
 \qed

\end{document}